\documentclass[reqno]{amsart}
\usepackage{amssymb,latexsym,amsmath,amsthm,enumerate,amsbsy}
\usepackage[mathscr]{eucal}
\usepackage{framed,color,graphicx}
\usepackage{mathrsfs}%花体字母
\usepackage{cite}%文献连接
\usepackage[all]{xy}
\usepackage{tikz}
\usetikzlibrary{shadows, shadings, calc}

\makeatletter
\@namedef{subjclassname@2020}{\textup{2020} Mathematics Subject Classification}
\makeatother

\usetikzlibrary{positioning,decorations.pathreplacing,patterns,decorations.pathmorphing}
\tikzset{%
element/.style={draw, shape=circle, fill=white, inner sep=1.4pt}
}
\DeclareSymbolFont{bbold}{U}{bbold}{m}{n}
\DeclareSymbolFontAlphabet{\mathbbold}{bbold}

\theoremstyle{plain}
\newtheorem{theorem}{Theorem}[section]
\newtheorem{lemma}[theorem]{Lemma}
\newtheorem{corollary}[theorem]{Corollary}
\newtheorem{proposition}[theorem]{Proposition}

\theoremstyle{definition}

\newtheorem{remark}[theorem]{Remark}

\newcommand{\bp}{\mathbf{p}}
\newcommand{\bq}{\mathbf{q}}
\newcommand{\br}{\mathbf{r}}
\newcommand{\bs}{\mathbf{s}}
\newcommand{\bt}{\mathbf{t}}
\newcommand{\bu}{\mathbf{u}}

\begin{document}
\title[Nonfinitely based ai-semiring not SNFB, limit variety]
{A nonfinitely based additively idempotent semiring that is not strongly nonfinitely based and generates a limit variety}

\author{Xiaolei Shao}
\address{School of Mathematics, Northwest University, Xi'an, 710127, Shaanxi, P.R. China}
\email{xiaoleishao@yeah.net}

\author{Miaomiao Ren}
\address{School of Mathematics, Northwest University, Xi'an, 710127, Shaanxi, P.R. China}
\email{miaomiaoren@yeah.net}

\author{Zidong Gao}
\address{School of Mathematics, Northwest University, Xi'an, 710127, Shaanxi, P.R. China}
\email{zidonggao@yeah.net}

%\thanks{Miaomiao Ren, corresponding author, is supported by National Natural Science Foundation of China (12371024, 12571020).}

\begin{abstract}
We present an explicit infinite equational basis for the six-element additively idempotent semiring $TR_6$
and prove that $TR_6$ is nonfinitely based.
We also give a complete description of the subvariety lattice of the variety generated by $TR_6$,
showing that it forms a four-element chain.
Our results demonstrate that the variety generated by $TR_6$ is a limit variety: it is itself nonfinitely based,
yet all of its proper subvarieties are finitely based.
This provides a new limit variety of additively idempotent semirings, distinct from all previously known ones.
In fact, $\mathsf{V}(TR_6)$ is the first explicit limit subvariety of the variety generated by the max-plus algebra $\mathbf{N}$.
Moreover, $TR_6$ is not strongly nonfinitely based: it belongs to a finitely based variety generated by a finite additively idempotent semiring.
Together with the six-element additively idempotent semiring $SR_6$,
these are the first two finite additively idempotent semirings that are nonfinitely based but not strongly nonfinitely based.
Finally, we study the variety generated by $SR_6$ and $TR_6$,
showing that it is nonfinitely based and has exactly nine subvarieties,
four of which are nonfinitely based and the remaining five are finitely based.
\end{abstract}

\maketitle

\section{Introduction}
A \emph{variety} is a class of algebras closed under taking subalgebras, homomorphic images, and arbitrary direct products.
By Birkhoff's celebrated theorem, a class of algebras is a variety if and only if it is an \emph{equational class};
that is, the class of all algebras satisfying a certain set of identities; such a set is an \emph{equational basis} of the variety.
A variety is \emph{finitely based} if it admits a finite equational basis; otherwise, it is \emph{nonfinitely based}.
An algebra $A$ is called finitely based (resp., nonfinitely based)
if the variety $\mathsf{V}(A)$ it generates is finitely based (resp., nonfinitely based).
If $\Sigma$ is an equational basis of the variety $\mathsf{V}(A)$,
we also say that $\Sigma$ is an equational basis of $A$.

The \emph{finite basis problem} for a class of algebras, one of the central problems in universal algebra,
concerns the classification of its members according to whether they are finitely based.
This problem has been intensively studied for groups, rings, semigroups, and semirings; see, for example,
\cite{neumann11, lee23, Volkov2001, JacksonRenZhao2022, Sapir2014}.

A variety is \emph{hereditarily finitely based} if all of its subvarieties are finitely based.
A variety is a \emph{limit variety} if it is itself nonfinitely based, but every proper subvariety of it is finitely based.
Equivalently, limit varieties are precisely the minimal nonfinitely based varieties.
By Zorn's lemma, every nonfinitely based variety contains a limit variety;
consequently, a variety is hereditarily finitely based if and only if it contains no limit subvarieties.
Thus the classification of hereditarily finitely based varieties is essentially equivalent to the classification of limit varieties.

The study of limit varieties is of intrinsic interest, as observed by Kharlampovich and Sapir~\cite{KharlampovichSapir1995},
who noted that problems concerning the boundary between finitely and nonfinitely based varieties are particularly appealing.
Nonetheless, limit varieties are notoriously difficult to find.
Lee and Volkov~\cite{LeeVolkov2011} emphasized that classifying limit varieties presents a number of challenges,
and that even producing a single concrete example can be a difficult task.

To date, no explicit example of a limit variety of groups has been found,
although the existence of infinitely many such varieties was established in \cite{Newman1971}.
This remains one of the central open problems in the theory of group varieties.
Within semigroups, however, several explicit examples are known (see~\cite{GusevSapir2022, Jackson2005, gulz, OSapir2023}).
In this paper, we present a new limit variety of additively idempotent semirings.

An \emph{additively idempotent semiring} (or an \emph{ai-semiring}, for short) is
an algebra $(S,+,\cdot)$ such that the additive reduct $(S,+)$ is a commutative idempotent semigroup,
the multiplicative reduct $(S,\cdot)$ is a semigroup, and the distributive laws
\[
x(y+z)\approx xy+xz,\qquad (x+y)z\approx xz+yz
\]
hold. The set of nonnegative integers, with $\max$ as addition and ordinary addition as multiplication,
forms an ai-semiring, known as the max-plus algebra restricted to the nonnegative integers.
We denote this ai-semiring by $\mathbf{N}$.
Such algebras arise naturally in various areas of mathematics and have important applications in
algebraic geometry, tropical geometry, information science, and theoretical computer science; see \cite{ConnesConsani2021, MaclaganSturmfels2015, Glazek2001, Golan1992}.

Let $S$ be an ai-semiring. The relation $\leq$ on $S$ given by
\[
a \leq b \Leftrightarrow a+b=b
\]
is a partial order; moreover, $(S, \leq)$ is an upper semilattice, with $a+b$ serving as the supremum of $a$ and $b$.
It is also compatible with multiplication, and for this reason ai-semirings are often called \emph{semilattice-ordered semigroups}.
Unless otherwise stated, all order-theoretic statements refer to this order.

Over the past two decades or so, the finite basis problem for ai-semirings has seen considerable progress; see \cite{aei, Dolinka2007, JacksonRenZhao2022, Jackson2008, ZhaoEtAl2020, GaoRen2025, RenLiuZengChen2025, RenLiuYueChen2026, YueRenZengShao2025, WuRenZhao2024, Shaprynskii2024, GhoshPastijnZhao2005,Pastijn2005,RenZhao2016,RenZhaoWang2017,DolinkaGusevVolkov2025, RenJacksonZhaoLei2023, yrg, yr2601}.
In particular, Aceto et al.~\cite{aei} showed that the max-plus algebra is nonfinitely based.
Dolinka~\cite{Dolinka2007} found the first example of a nonfinitely based finite ai-semiring,
while Jackson~\cite{Jackson2008} subsequently presented infinitely many such examples.
Pastijn et al.~\cite{GhoshPastijnZhao2005, Pastijn2005} and
Ren et al.~\cite{RenZhao2016, RenZhaoWang2017} showed that every ai-semiring satisfying the identity $x^3 \approx x$ is finitely based.
Dolinka et al.~\cite{DolinkaGusevVolkov2025}
solved the finite basis problem for the endomorphism semirings of finite semilattices.
Jackson et al.~\cite{JacksonRenZhao2022} and Zhao et al.~\cite{ZhaoEtAl2020} completed the classification of
ai-semirings of order at most three with respect to the finite basis property.
Recently, the finite basis problem for four-element ai-semirings has been investigated in \cite{RenLiuZengChen2025, RenLiuYueChen2026, YueRenZengShao2025, WuRenZhao2024, Shaprynskii2024, yr2601, yrg}.

In contrast, the study of limit varieties of ai-semirings has progressed more slowly;
it was not until recently that substantial advances were made.
Ren et al.~\cite{RenJacksonZhaoLei2023} provided the first explicit examples of limit varieties of ai-semirings.
They gave a complete characterization of limit varieties generated by flat extensions of finite groups, yielding an infinite family of such varieties, each generated by a finite flat semiring (for more on flat semirings, see~\cite{Jackson2008}).
In addition to this infinite family, they exhibited several other examples: an ad hoc example generated by a finite flat semiring, a continuum-sized family with no finite generator, and two further ad hoc examples.
Another example was later given by Gao and Ren~\cite{GaoRen2025},
who constructed a limit variety generated by all acyclic graph semirings; this variety is not finitely generated.

Very recently, Lyu et al.~\cite{LyuRenYue2026} showed that the variety generated by the six-element commutative ai-semiring $SR_6$ is a limit variety (see Table~\ref{tab:SR6} for its Cayley tables).
More precisely, they established a sufficient condition for an ai-semiring to be nonfinitely based,
and applied it to prove that $SR_6$ is nonfinitely based, providing an explicit infinite equational basis for it.
They also described the subvariety lattice of $\mathsf{V}(SR_6)$ in full detail, showing that it forms a four-element chain; in particular, every proper subvariety of $\mathsf{V}(SR_6)$ is finitely based.
Thus the variety $\mathsf{V}(SR_6)$ is a limit variety with a completely described subvariety lattice.

\begin{table}[htbp]
\centering
\caption{The Cayley tables of $SR_6$}\label{tab:SR6}
\begin{tabular}{c|cccccc}
$+$               &1 & 2 & 3 & 4 & 5 & 6\\
\hline
                 1& 1 & 1 & 1 & 1 & 1 & 1 \\
                 2&1 & 2 & 1 & 1 & 2 & 1\\
                 3&1 & 1 & 3 & 1 & 1 & 1\\
                 4&1 & 1 & 1 & 4 & 1 & 4\\
                 5&1 & 2 & 1 & 1 & 5 & 1\\
                 6&1 & 1 & 1 & 4 & 1 & 6
\end{tabular}
\qquad\qquad
\begin{tabular}{c|cccccc}
$\cdot$               &1 & 2 & 3 & 4 & 5 & 6\\
\hline
                   1&1 & 1 & 1 & 1 & 1 & 1 \\
                   2&1 & 1 & 1 & 1 & 1 & 3\\
                   3&1 & 1 & 1 & 1 & 1 & 1\\
                   4&1 & 1 & 1 & 1 & 3 & 1\\
                   5&1 & 1 & 1 & 3 & 1 & 3\\
                   6&1 & 3 & 1 & 1 & 3 & 1
\end{tabular}
\end{table}

A finite ai-semiring is \emph{inherently nonfinitely based} (INFB) if it belongs to no finitely based locally finite variety;
that is, every locally finite variety containing it is nonfinitely based.
A finite ai-semiring is \emph{strongly nonfinitely based} (SNFB) if
it belongs to no finitely based finitely generated variety;
that is, every finitely generated variety containing it is nonfinitely based.

It is immediate from the definitions that every INFB ai-semiring is SNFB, and every SNFB ai-semiring is nonfinitely based.
The converse of the first implication does not hold: a SNFB ai-semiring need not be INFB (see \cite[Example~6.7]{JacksonRenZhao2022}).
Moreover, the converse of the second implication also fails: a nonfinitely based finite ai-semiring need not be SNFB.
Prior to this paper, however, no such example was known.
Although Lyu et al.~\cite{LyuRenYue2026} showed that $SR_6$ generates a limit variety,
they did not establish whether $SR_6$ is SNFB.
Here we prove that it is not, and we also provide a second such example in Theorem~\ref{thm:MR6}.
We now give the first such example.

\begin{theorem}
The ai-semiring $SR_6$ is nonfinitely based, but not strongly nonfinitely based.
\end{theorem}
\begin{proof}
By \cite[Corollary 4.7]{LyuRenYue2026}, $SR_6$ is nonfinitely based.
It remains to show that $SR_6$ is not strongly nonfinitely based.
Let $S_{(4, 413)}$ denote the four-element ai-semiring whose Cayley tables are given in Table~\ref{S4-413}.
By \cite[Proposition~5.4]{YueRenZengShao2025},
the ai-semiring variety $\mathsf{V}(S_{(4, 413)})$ is defined by the following five identities:
\[
xy \approx yx,\quad x\preceq x^2,\quad x+xy\approx x^2+xy,\quad x_4\preceq x_1x_2x_3,\quad xy \preceq x^2+yz.
\]
It is routine to check that $SR_6$ satisfies these identities,
and so $SR_6$ belongs to the finitely based ai-semiring variety $\mathsf{V}(S_{(4, 413)})$, which is finitely generated.
Therefore, $SR_6$ is not strongly nonfinitely based.
\end{proof}

\begin{table}[htbp]
\centering
\small
\caption{The Cayley tables of \(S_{(4,413)}\)}\label{S4-413}
\renewcommand{\arraystretch}{1.08}
\setlength{\tabcolsep}{4.8pt}
\begin{tabular}{c|cccc}
$+$ & 1 & 2 & 3 & 4 \\ \hline
1 & 1 & 1 & 1 & 1 \\
2 & 1 & 2 & 3 & 4 \\
3 & 1 & 3 & 3 & 1 \\
4 & 1 & 4 & 1 & 4
\end{tabular}
\qquad\qquad
\begin{tabular}{c|cccc}
$\cdot$ & 1 & 2 & 3 & 4 \\ \hline
1 & 1 & 1 & 1 & 1 \\
2 & 1 & 3 & 1 & 3 \\
3 & 1 & 1 & 1 & 1 \\
4 & 1 & 3 & 1 & 1
\end{tabular}
\end{table}

\begin{table}[htbp]
\centering
\caption{The Cayley tables of $TR_6$}\label{cayleyMR6}
\small
\renewcommand{\arraystretch}{1.08}
\setlength{\tabcolsep}{4.8pt}
\begin{tabular}{c|cccccc}
$+$&1&2&3&4&5&6\\ \hline
1&1&1&1&1&1&1\\
2&1&2&3&3&2&3\\
3&1&3&3&3&3&3\\
4&1&3&3&4&3&4\\
5&1&2&3&3&5&3\\
6&1&3&3&4&3&6
\end{tabular}
\hspace{1.2cm}
\begin{tabular}{c|cccccc}
$\cdot$&1&2&3&4&5&6\\ \hline
1&1&1&1&1&1&1\\
2&1&1&1&1&1&3\\
3&1&1&1&1&1&1\\
4&1&1&1&1&3&1\\
5&1&1&1&3&1&3\\
6&1&3&1&1&3&1
\end{tabular}
\end{table}

Let $TR_6$ denote the six-element commutative ai-semiring $\{1,2,3,4,5,6\}$ whose Cayley tables are given in Table~\ref{cayleyMR6}.
Then $TR_6$ and $SR_6$ share the same multiplicative reduct, but have different additive reducts.
In both of them, the element $1$ is the multiplicative zero element and the additive maximum element.
These similarities, however, do not imply any inclusion between the varieties they generate.
For instance, the inequality $x \preceq xy$ is satisfied by $TR_6$ but not by $SR_6$, while the identity $x^2\approx x+xy$ holds in $SR_6$ but fails in $TR_6$.
Thus the varieties $\mathsf{V}(TR_6)$ and $\mathsf{V}(SR_6)$ are incomparable,
but they have the same finite basis property.
We now present the second example of a finite ai-semiring that is nonfinitely based but not SNFB.

\begin{theorem}\label{thm:MR6}
The ai-semiring $TR_6$ is nonfinitely based, but not strongly nonfinitely based.
\end{theorem}
\begin{proof}
We first show that $TR_6$ is not strongly nonfinitely based;
its nonfinite basis property will be established in Proposition~\ref{nfb}.
Let $S_{53}$ denote the three-element ai-semiring whose Cayley tables are given in Table~\ref{tb5301}.
By \cite[Proposition 7]{ZhaoEtAl2020}, the variety $\mathsf{V}(S_{53})$ is defined by the following four identities:
\[
xy \approx yx, \quad xy+y^2 \approx x+y^2, \quad x\preceq xy, \quad xyz\approx xy+yz+xz.
\]
It is easy to check that these identities hold in $TR_6$,
so $TR_6$ belongs to the finitely based ai-semiring variety $\mathsf{V}(S_{53})$, which is finitely generated.
Therefore, $TR_6$ is not strongly nonfinitely based.
\end{proof}

\begin{table}[ht]
\caption{The Cayley tables of $S_{53}$} \label{tb5301}
\begin{tabular}{c|ccc}
$+$      &$1$ &$2$ &$3$\\
\hline
$1$      &$1$ &$1$ &$3$\\
$2$      &$1$ &$2$ &$3$\\
$3$      &$3$ &$3$ &$3$\\
\end{tabular}\qquad
\begin{tabular}{c|ccc}
$\cdot$      &$1$ &$2$ &$3$\\
\hline
$1$      &$3$ &$1$ &$3$\\
$2$      &$1$ &$2$ &$3$\\
$3$      &$3$ &$3$ &$3$\\
\end{tabular}
\end{table}

We shall show that $\mathsf{V}(TR_6)$ is a new limit variety, distinct from all previously known ones.
Like $\mathsf{V}(SR_6)$, its subvariety lattice forms a four-element chain;
the bottom two subvarieties coincide with those of $\mathsf{V}(SR_6)$, while the upper parts are different.

The max-plus algebra $\mathbf{N}$ is a canonical infinite ai-semiring.
It is nonfinitely based \cite[Theorem~4.1]{aei}, and
the variety $\mathsf{V}(\mathbf{N})$ has uncountably many subvarieties \cite[Proposition~4.8]{Gao2026},
yet its subvarieties remain largely unexplored.
In particular, no explicit limit subvariety of $\mathsf{V}(\mathbf{N})$ was known before this work.
We shall show that $TR_6$ lies in $\mathsf{V}(\mathbf{N})$,
and so $\mathsf{V}(TR_6)$ is the first explicit limit subvariety of $\mathsf{V}(\mathbf{N})$.
This addresses an open direction of Ren et al.~\cite[Section~6]{RenJacksonZhaoLei2023},
who raised the question of exploring subvarieties of $\mathsf{V}(\mathbf{N})$ in search of limit varieties.

Thus $SR_6$ and $TR_6$ share several striking features: both are limit varieties with four-element subvariety lattices,
and both are the first examples of finite ai-semirings that are nonfinitely based but not SNFB.
Given these parallels, it is natural to study the variety they generate together.
We shall do so in detail, completely classifying its subvarieties with respect to the finite basis property.

The paper is organized as follows.
In Section~2, we collect the necessary preliminaries.
In Section~3, we provide an explicit infinite equational basis for $TR_6$ and prove that it is nonfinitely based.
In Section~4, we describe the subvariety lattice of $\mathsf{V}(TR_6)$ and show that $\mathsf{V}(TR_6)$ is a limit variety.
In Section~5, we provide an explicit infinite equational basis for the variety $\mathsf{V}(SR_6, TR_6)$ generated by $SR_6$ and $TR_6$
and prove that it is nonfinitely based.
In Section~6, we describe the subvariety lattice of $\mathsf{V}(SR_6, TR_6)$.
Finally, Section~7 concludes the paper with some remarks and open problems.

\section{Preliminaries}
In this section, we introduce the basic terminology, notation, and tools that will be used throughout the paper.
Since our main object of study, $TR_6$, is commutative, we focus primarily on notions and facts related to commutative ai-semirings.
Much of the material is drawn from the preliminary sections of \cite{LyuRenYue2026, yrg}.

Let $X$ be a countably infinite set of variables, and let $X_c^+$ denote the free commutative semigroup over $X$.
Thus $X_c^+$ consists of all nonempty words over $X$, with two words identified
if one can be obtained from the other by permuting the variables.
Equivalently, a word is determined up to the multiplicities of its variables.
We also write $X_c^*$ for the free commutative monoid over $X$.

A \emph{commutative ai-semiring term} (or simply a \emph{term}) over $X$ is defined as a finite nonempty set of words in $X_c^+$. In the sequel, terms are denoted by bold lowercase letters $\mathbf{u},\mathbf{v},\mathbf{w},\ldots$, while ordinary lowercase letters $x,y,z,\ldots$ represent variables. A term is represented as a formal sum of its elements. Specifically,
\(
\mathbf{w}=\mathbf{u}_{1}+\mathbf{u}_{2}+\cdots+\mathbf{u}_{n}
\)
indicates that $\mathbf{w}=\{\mathbf{u}_{1},\mathbf{u}_{2},\ldots,\mathbf{u}_{n}\}$. The order of the summands in the formal sum is irrelevant, and multiple occurrences of the same word are collapsed into a single occurrence. Two terms are equal if and only if their underlying sets coincide.

Let $P_f(X_c^+)$ denote the set of all terms over $X$.
This set forms a commutative ai-semiring with union as addition and elementwise product as multiplication.
By \cite[Theorem~2.5]{KurilPolak2005},
$P_f(X_c^+)$ is the free commutative ai-semiring over $X$ in the variety of all commutative ai-semirings.
A \emph{commutative ai-semiring substitution} (or simply a \emph{substitution})
is a semiring homomorphism from $P_f(X_c^+)$ to itself.

Let $\mathbf{u}$ and $\mathbf{v}$ be terms. We say that $\mathbf{u}$ is a \emph{subterm} of $\mathbf{v}$ if there exist terms $\mathbf{p}_1$ and $\mathbf{p}_2$ such that
\[
\mathbf{v}=\mathbf{p}_1\mathbf{u}+\mathbf{p}_2.
\]
Here $\mathbf{p}_1$ may be the empty word (acting as the multiplicative identity),
and $\mathbf{p}_2$ may be the empty set (acting as the additive identity).
In particular, if $\mathbf{p}_1$ is empty, then $\mathbf{u}$ is an \emph{additive subterm} of $\mathbf{v}$.
The term $\mathbf{v}$ is called $\mathbf{u}$-\emph{free} if for every substitution $\varphi$,
the term $\varphi(\mathbf{u})$ is not a subterm of $\mathbf{v}$.

\begin{lemma}\label{freelemma}
Let $\mathbf{u}$, $\mathbf{v}$, and $\mathbf{w}$ be terms.
If $\mathbf{v}$ is $\mathbf{u}$-free and $\mathbf{u}$ is a subterm of $\mathbf{w}$,
then $\mathbf{v}$ is also $\mathbf{w}$-free.
\end{lemma}

\begin{proof}
This follows directly from the definitions of subterm and freeness.
\end{proof}

A \emph{commutative ai-semiring identity} (or simply an \emph{identity}) is a formal expression of the form $\mathbf{u}\approx\mathbf{v}$, where $\mathbf{u}$ and $\mathbf{v}$ are terms. Let $S$ be a commutative ai-semiring and $\mathbf{u}\approx\mathbf{v}$ an identity. We say that $S$ satisfies $\mathbf{u}\approx\mathbf{v}$, or that $\mathbf{u}\approx\mathbf{v}$ holds in $S$, if $\varphi(\mathbf{u})=\varphi(\mathbf{v})$ for every semiring homomorphism $\varphi:P_f(X_c^+)\to S$.

The following result concerns the equational logic of commutative ai-semirings.
It can be seen as a direct specialization of a standard universal-algebraic fact
(see \cite[Exercise II.14.11]{BurrisSankappanavar1981}) to commutative ai-semirings.

\begin{lemma}\label{eqyield}
Let $\Sigma$ be a set of identities and let $\mathbf{u}\approx\mathbf{v}$ be a nontrivial identity. Then $\mathbf{u}\approx\mathbf{v}$ is derivable from $\Sigma$ if and only if there exist terms $\mathbf{t}_1,\mathbf{t}_2,\ldots,\mathbf{t}_n\in P_f(X_c^+)$ such that $\mathbf{u}=\mathbf{t}_1$, $\mathbf{v}=\mathbf{t}_n$ and, for each $1\leq i<n$, there are terms $\mathbf{p}_i,\mathbf{r}_i,\mathbf{s}_i,\mathbf{s}_i'\in P_f(X_c^+)$ and a substitution $\varphi_i:P_f(X_c^+)\to P_f(X_c^+)$ such that
\[
\mathbf{t}_i=\mathbf{p}_i\varphi_i(\mathbf{s}_i)+\mathbf{r}_i,\qquad
\mathbf{t}_{i+1}=\mathbf{p}_i\varphi_i(\mathbf{s}_i')+\mathbf{r}_i,
\]
where $\mathbf{s}_i\approx\mathbf{s}_i'\in\Sigma$ or $\mathbf{s}_i'\approx\mathbf{s}_i\in\Sigma$, $\mathbf{p}_i$ may be the empty word (acting as the multiplicative identity), and $\mathbf{r}_i$ may be the empty set (acting as the additive identity).
\end{lemma}

The following lemma, which is the compactness theorem of equational logic (see \cite[Exercise~II.14.10]{BurrisSankappanavar1981}),
will be our main tool for establishing nonfinite basedness of several varieties in the sequel.

\begin{lemma}\label{compactness}
Let $\mathcal V$ be a commutative ai-semiring variety. If $\mathcal V$ is finitely based, then every equational basis of $\mathcal V$ contains a finite subset that also defines $\mathcal V$. Equivalently, if $\mathcal V$ has an infinite equational basis none of whose finite subsets defines $\mathcal V$, then $\mathcal V$ is nonfinitely based.
\end{lemma}

We denote by $\mathbf{u}\preceq\mathbf{v}$ (or equivalently $\mathbf{v}\succeq\mathbf{u}$) the identity $\mathbf{v}\approx\mathbf{v}+\mathbf{u}$, which we refer to as a \emph{commutative ai-semiring inequality} (or simply an \emph{inequality}). It is routine to verify that a commutative ai-semiring $S$ satisfies the inequality $\mathbf{u}\preceq\mathbf{v}$ if and only if for every semiring homomorphism $\varphi:P_f(X_c^+)\to S$, we have $\varphi(\mathbf{u})\leq\varphi(\mathbf{v})$. Consequently, $S$ satisfies an identity $\mathbf{u}\approx\mathbf{v}$ precisely when it satisfies both inequalities $\mathbf{u}\preceq\mathbf{v}$ and $\mathbf{v}\preceq\mathbf{u}$.
Therefore, $\mathbf{u}\approx\mathbf{v}$ is equivalent to the inequalities $\mathbf{u}\preceq\mathbf{v}$ and $\mathbf{v}\preceq\mathbf{u}$.

Now let $\Sigma$ be a set of identities, and let $\mathbf{u}\approx\mathbf{v}$ be an identity such that
\[
\mathbf{u}=\mathbf{u}_{1}+\cdots+\mathbf{u}_{k},\qquad
\mathbf{v}=\mathbf{v}_{1}+\cdots+\mathbf{v}_{\ell},
\]
where $\mathbf{u}_i,\mathbf{v}_j\in X_c^+$ for $1\leq i\leq k$ and $1\leq j\leq\ell$. One readily checks that the commutative ai-semiring variety defined by $\mathbf{u}\approx\mathbf{v}$ coincides with the commutative ai-semiring variety defined by the inequalities
\[
\mathbf{u}_i\preceq\mathbf{v},\qquad \mathbf{v}_j\preceq\mathbf{u}
\qquad(1\leq i\leq k,\ 1\leq j\leq\ell).
\]
Consequently, to prove that $\mathbf{u}\approx\mathbf{v}$ is derivable from $\Sigma$, it suffices to show that for every $i$ and $j$, the inequalities $\mathbf{u}_i\preceq\mathbf{v}$ and $\mathbf{v}_j\preceq\mathbf{u}$ can be derived from $\Sigma$. In view of this, we always restrict our attention to inequalities of the form $\mathbf{q}\preceq\mathbf{u}$, where $\mathbf{q}$ is a word and $\mathbf{u}$ is a term.

For a \emph{word} $\mathbf{p}$, let $\ell(\mathbf{p})$ denote the \emph{length} of $\mathbf{p}$;
that is, the number of variables occurring in $\mathbf{p}$ counting multiplicities,
and let $c(\mathbf{p})$ denote the \emph{content} of $\mathbf{p}$;
that is, the set of variables occurring in $\mathbf{p}$.
For a term $\mathbf{t}$, let $c(\mathbf{t})$ denote its content; that is,
\[
c(\mathbf{t})=\bigcup_{\bp\in \mathbf{t}} c(\bp).
\]
For an integer $k\geq1$, let $L_k(\mathbf{t})$ denote the term that is the sum of all words in $\mathbf{t}$ of length $k$.

\section{An equational basis for $TR_6$ and nonfinite basedness}
In this section, we present an explicit infinite equational basis for $TR_6$,
and then prove that $TR_6$ is nonfinitely based.
To this end, we first introduce some terms.

For each integer $n\geq1$, let $\mathbf{u}^{(n)}$ denote the term
\[
\mathbf{u}^{(n)}=x_1x_2+x_2x_3+\cdots+x_{2n}x_{2n+1}+x_{2n+1}x_1.
\]
These terms have played an important role in establishing the nonfinite basis property
of ai-semirings (see \cite{WuRenZhao2024, LyuRenYue2026, yr2601}).
We denote by $\sigma_n$ the inequality $x \preceq \mathbf{u}^{(n)}$.

\begin{lemma}\label{R6identities}
The ai-semiring $TR_6$ satisfies the identities $\sigma_n$ for all $n \geq 1$, together with
\begin{align}
&x\preceq y^2;  \label{eq2}\\
&x\preceq xy;  \label{eq3}\\
&x^2\approx x_1x_2x_3. \label{eq1}
\end{align}
\end{lemma}
\begin{proof}
It is straightforward to check that $TR_6$ satisfies the identities~\eqref{eq2}, \eqref{eq3}, and \eqref{eq1}.
So the main task is to prove that the inequality $\sigma_n$ holds in $TR_6$ for all $n\geq1$.
Indeed, let $n\geq 1$ be an integer, and
let $\varphi\colon P_f(X_c^+) \to TR_6$ be an arbitrary semiring homomorphism.
We claim that $\varphi(\mathbf{u}^{(n)})=1$.
Indeed, from the multiplicative Cayley table of $TR_6$ one sees that the product of any two elements is either $1$ or $3$;
consequently, $\varphi(\mathbf{w})=1$ or $3$ for every word $\mathbf{w}\in\mathbf{u}^{(n)}$.
We shall rule out the possibility that $\varphi(\mathbf{w})=3$ for every $\mathbf{w}\in\mathbf{u}^{(n)}$.

Suppose, to the contrary, that $\varphi(x_ix_{i+1})=3$ for all $i=1,2,\ldots,2n+1$,
where indices are taken modulo $2n+1$.
Then $\varphi(x_i)\varphi(x_{i+1})=3$ for each $i$. From the multiplicative Cayley table of $TR_6$, the only unordered pairs $\{a,b\}$ with $ab=3$ are
\[
\{2,6\}, \quad \{4,5\}, \quad \{5,6\}.
\]
Thus each $\varphi(x_i)$ must belong to $\{2,4,5,6\}$, and consecutive elements must form one of these three unordered pairs.

Define a graph $\mathbb{G}$ with vertex set $\{2,4,5,6\}$ and edge set
\[
\bigl\{\{2,6\},\{4,5\},\{5,6\}\bigr\}.
\]
Observe that $\mathbb{G}$ is a path of length $3$, and has no cycles.
Then every closed walk in $\mathbb{G}$ has even length.
Now consider the closed walk in $\mathbb{G}$:
\[
\varphi(x_1), \varphi(x_2), \ldots, \varphi(x_{2n+1}), \varphi(x_1),
\]
which has length $2n+1$. Since $2n+1$ is odd, such a closed walk cannot exist in $\mathbb{G}$.
Hence $\varphi(x_ix_{i+1})=1$ for some $i$.
Since $1$ is the additive maximum element of $TR_6$,
it follows immediately that $\varphi(\mathbf{u}^{(n)})=1$, and so $\varphi(x) \leq \varphi(\mathbf{u}^{(n)})$.
Therefore, $\sigma_n$ holds in $TR_6$.
\end{proof}

\begin{remark}
By Lemma~\ref{R6identities}, in any ai-semiring $S$ belonging to $\mathsf{V}(TR_6)$,
the square of any element, as well as the product of three or more elements,
is both the additive maximum element and the multiplicative zero element.
Moreover, $a \leq ab$ holds for all elements $a$ and $b$ of $S$; that is,
the product of two elements is greater than or equal to each of its factors.
These facts will be used several times later in the paper.
\end{remark}

\begin{remark}\label{remark26082910}
The ai-semiring $SR_6$ also satisfies the identities $\sigma_n$ for all $n \geq 1$.
Indeed, the proof that $TR_6$ satisfies $\sigma_n$ (Lemma~\ref{R6identities}) relies only on the multiplicative reduct of $TR_6$ and the fact that $1$ is its additive maximum element.
Since $SR_6$ and $TR_6$ share the same multiplicative reduct and $1$ is also the additive maximum element in $SR_6$, the same argument applies to $SR_6$.
\end{remark}

For each term $\mathbf{t}$,
let $\mathbb{G}_{\mathbf{t}}$ be the graph with vertex set
$V(\mathbb{G}_{\mathbf{t}})=c(L_2(\mathbf{t}))$
and edge set
$
E(\mathbb{G}_{\mathbf{t}})
=\bigl\{\{x,y\}\mid xy\in L_2(\mathbf{t})\bigr\}.
$
Then $\mathbb{G}_{\mathbf{t}}$ may contain loops (corresponding to words of the form $x^2$), but it has no multiple edges.
This graph will play a key role in the analysis that follows.
In particular, $\mathbb{G}_{\mathbf{u}^{(n)}}$ is an odd cycle of length $2n+1$, with vertices $x_1,\ldots,x_{2n+1}$ and edges $\{x_i,x_{i+1}\}$, where the indices are taken modulo $2n+1$.

We shall show that the identities
\eqref{eq2}, \eqref{eq3}, \eqref{eq1}, and $\sigma_n$ $(n \geq 1)$ form a countably infinite equational basis for $TR_6$.
For this purpose, we characterize the inequalities satisfied by $TR_6$.
A \emph{linear word} is a word that contains no repeated variables.

\begin{proposition}\label{wordprop}
Let $\mathbf{q}\preceq\mathbf{u}$ be a nontrivial inequality such that $\mathbf{u}=\mathbf{u}_1+\mathbf{u}_2+\cdots+\mathbf{u}_n$ with $\mathbf{u}_i,\mathbf{q}\in X_c^+$ for $1\leq i\leq n$. Then $\mathbf{q}\preceq\mathbf{u}$ is satisfied by $TR_6$ if and only if $\mathbf{u}$ and $\mathbf{q}$ satisfy one of the following conditions:
\begin{enumerate}[$(\rm i)$]
\item $\ell(\mathbf{u}_i)\geq3$ for some $\mathbf{u}_i\in\mathbf{u}$;
\item $\mathbf{u}$ contains a non-linear word;
\item the graph $\mathbb{G}_{\mathbf{u}}$ contains an odd cycle;
\item $\ell(\mathbf{q})=1$ and $c(\mathbf{q})\subseteq c(\mathbf{u})$.
\end{enumerate}
\end{proposition}
\begin{proof}
By Lemma~\ref{R6identities}, $TR_6$ satisfies the identities \eqref{eq2}, \eqref{eq3}, and \eqref{eq1}.
Suppose that one of the conditions {\rm (i)}--{\rm (iii)} holds.
Let $\varphi\colon P_f(X_c^+) \to TR_6$ be an arbitrary semiring homomorphism.
If {\rm (i)} or {\rm (ii)} holds, then by the identities~\eqref{eq2} and \eqref{eq1}, $\varphi(\mathbf{u})=1$,
and so $\varphi(\mathbf{q})\leq\varphi(\mathbf{u})$.
If {\rm (iii)} holds, then there exists an additive subterm $\mathbf{v}$ of $\mathbf{u}$ such that
$\mathbb{G}_\mathbf{v}$ is an odd cycle.
By the same graph argument as in the proof of Lemma~\ref{R6identities},
$\varphi(\mathbf{v})=1$, and so
\[
\varphi(\mathbf{u}) = \varphi(\mathbf{u}+\mathbf{v}) = \varphi(\mathbf{u})+\varphi(\mathbf{v})=\varphi(\mathbf{u})+1=1
\]
and so $\varphi(\mathbf{q})\leq\varphi(\mathbf{u})$.
	
Now suppose that {\rm (iv)} holds.
Since the inequality $\mathbf{q}\preceq\mathbf{u}$ is nontrivial,
there exists a word $\mathbf{u}_i\in L_{\geq 2}(\mathbf{u})$ such that $c(\mathbf{q})\subseteq c(\mathbf{u}_i)$,
and so $\mathbf{u}_i=\mathbf{q}\mathbf{p}$ for some $\mathbf{p}\in X_c^+$.
By the identity~\eqref{eq3}, we have
\[
\mathbf{u} \succeq \mathbf{u}_i=\mathbf{q}\mathbf{p} \succeq \mathbf{q}.
\]
Thus $TR_6$ satisfies $\mathbf{q}\preceq\mathbf{u}$.
	
Conversely, assume that $TR_6$ satisfies $\mathbf{q}\preceq\mathbf{u}$.
We need to show that one of the conditions {\rm (i)}--{\rm (iii)} holds.
Assume that the conditions {\rm (i)}, {\rm (ii)}, and {\rm (iii)} are all false.
We shall show that {\rm (iv)} must hold.
	
Under this assumption, $\mathbf{u}=L_1(\mathbf{u})+L_2(\mathbf{u})$, every word in $L_2(\mathbf{u})$ is linear,
and the graph $\mathbb{G}_{\mathbf{u}}$ contains no odd cycles.
Hence the graph $\mathbb{G}_{\mathbf{u}}$ is bipartite.
Consequently, there are disjoint sets $A$ and $B$ with $c(L_2(\mathbf{u}))=A \cup B$
such that every edge of $\mathbb{G}_{\mathbf{u}}$ has one endpoint in $A$ and the other in $B$;
that is, $(A,B)$ is a bipartition of $\mathbb{G}_{\mathbf{u}}$.
We may therefore write
\[
L_2(\mathbf{u})=x_1y_1+x_2y_2+\cdots+x_ry_r, \quad x_i\in A, \quad y_i\in B.
\]
	
Define a semiring homomorphism $\varphi\colon P_f(X_c^+)\to TR_6$ by: for every $t\in X$,
\[
\varphi(t)=
\begin{cases}
5, & t\in A,\\
6, & t\in B,\\
3, & t\in c(L_1(\bu))\setminus c(L_2(\bu)),\\
1, & \text{otherwise}.
\end{cases}
\]
Every word in $L_2(\mathbf{u})$ is sent to $5\cdot6=3$,
and every word in $L_1(\mathbf{u})$ is sent to $5$ or $6$ or $3$, which is less than or equal to $3$.
Thus
\[
\varphi(\mathbf{u})=\varphi(L_1(\mathbf{u})+L_2(\mathbf{u}))=\varphi(L_1(\mathbf{u}))+\varphi(L_2(\mathbf{u}))
=\varphi(L_1(\mathbf{u}))+3=3.
\]
Since $\mathbf{q}\preceq\mathbf{u}$ holds in $TR_6$, we have that $\varphi(\mathbf{q})\leq\varphi(\mathbf{u})=3$.
Hence $\varphi(\mathbf{q})\neq 1$.
By the definition of $\varphi$, $c(\mathbf{q})\subseteq c(\mathbf{u})$.
Moreover, every product of three or more elements and every square in $TR_6$ is equal to $1$.
Therefore, $\mathbf{q}$ is a linear word of length at most two.
	
Suppose that $\ell(\mathbf{q})=2$. Then $\varphi(\mathbf{q})=3$.
By the definition of $\varphi$, $\mathbf{q}=xy$ for some $x\in A$ and $y\in B$.	
Since $\mathbf{q}\preceq\mathbf{u}$ is nontrivial, $\mathbf{q}\notin L_2(\mathbf{u})$.
Define another semiring homomorphism $\psi \colon P_f(X_c^+)\to TR_6$ by setting, for every $t\in X$,
\[
\psi(t)=
\begin{cases}
2, & t=x,\\
4, & t=y,\\
5, & t\in A\setminus\{x\},\\
6, & t\in B\setminus\{y\},\\
3, & \text{otherwise}.
\end{cases}
\]
For each $x_iy_i\in L_2(\mathbf{u})$, the pair $(\psi(x_i),\psi(y_i))$ is one of $(2,6)$, $(5,4)$, or $(5,6)$;
the pair $(2,4)$ cannot occur because $xy=\mathbf{q}\notin L_2(\mathbf{u})$. In every case, $\psi(x_iy_i)=3$.
Every word in $L_1(\mathbf{u})$ is sent to an element, which is less than or equal to $3$.
Therefore $\psi(\mathbf{u})=3$, while
\[
\psi(\mathbf{q})=\psi(x)\psi(y)=2\cdot4=1\nleq3=\psi(\mathbf{u}).
\]
This contradicts $\mathbf{q}\preceq\mathbf{u}$. Hence $\ell(\mathbf{q})=1$.
Therefore {\rm (iv)} holds.
\end{proof}

We now present an infinite equational basis for $TR_6$.

\begin{proposition}\label{basisR6}
The commutative ai-semiring variety $\mathsf{V}(TR_6)$ is defined by the identities
\eqref{eq2}, \eqref{eq3}, \eqref{eq1}, and $\sigma_n$ $(n \geq 1)$.
\end{proposition}
\begin{proof}
By Lemma~\ref{R6identities}, $TR_6$ satisfies the identities
\eqref{eq2}, \eqref{eq3}, \eqref{eq1}, and $\sigma_n$ $(n \geq 1)$.
It remains to show that every inequality satisfied by $TR_6$ is derivable from
these identities.
Let $\mathbf{q}\preceq\mathbf{u}$ be such a nontrivial inequality, where $\mathbf{u}=\mathbf{u}_1+\mathbf{u}_2+\cdots+\mathbf{u}_n$
with $\mathbf{u}_i, \mathbf{q}\in X_c^+$ for $1\leq i\leq n$.
By Proposition~\ref{wordprop}, we consider the following four cases.

\textbf{Case 1.} $L_{\geq3}(\mathbf{u})\neq \emptyset$.
Then there exists a word $\mathbf{u}_i\in\mathbf{u}$ such that $\ell(\mathbf{u}_i)\geq3$. Thus
\[
\mathbf{u}\succeq \mathbf{u}_i\overset{\eqref{eq1}}{\approx} x^2
\overset{\eqref{eq2}}{\succeq}\mathbf{q}.
\]

\textbf{Case 2.} $\mathbf{u}$ contains a non-linear word.
Then there exists a word $\mathbf{u}_i \in \mathbf{u}$ such that $\mathbf{u}_i=x^2\bp$ for some $x\in X$ and some $\bp\in X_c^*$.
Thus
\[
\mathbf{u}\succeq \mathbf{u}_i= x^2\bp
\overset{\eqref{eq1}}{\approx}x^2
\overset{\eqref{eq2}}{\succeq}\mathbf{q}.
\]

\textbf{Case 3.} The graph $\mathbb{G}_{\mathbf{u}}$ contains an odd cycle. Then $L_2(\mathbf{u})$ contains
\[
x_1x_2,\ x_2x_3,\ldots,\ x_{2k}x_{2k+1},\ x_{2k+1}x_1
\]
for some $x_1,x_2,\ldots,x_{2k+1}\in X$ and $k\geq1$. Hence
\[
\mathbf{u}\succeq\mathbf{u}^{(k)}
\overset{\sigma_k}{\succeq}\mathbf{q}.
\]

\textbf{Case 4.} $\ell(\mathbf{q})=1$ and $c(\mathbf{q})\subseteq c(\mathbf{u})$.
Then there exists a word $\mathbf{u}_i\in\mathbf{u}$ such that $\mathbf{u}_i=\mathbf{q}\bp$ for some $\bp\in X_c^+$.
Thus
\[
\mathbf{u}\succeq \mathbf{u}_i=\mathbf{q}\bp \overset{\eqref{eq3}}{\succeq}\mathbf{q}.
\]

Therefore, in each case, the inequality $\mathbf{q}\preceq\mathbf{u}$
is derivable from the identities \eqref{eq2}, \eqref{eq3}, \eqref{eq1}, and $\sigma_n$ $(n \geq 1)$, completing the proof.
\end{proof}

We are now ready to prove the main result of this section.

\begin{proposition}\label{nfb}
The ai-semiring $TR_6$ is nonfinitely based.
\end{proposition}
\begin{proof}
Let $\Sigma$ denote the collection of the identities
\eqref{eq2}, \eqref{eq3}, \eqref{eq1}, and $\sigma_n$ $(n \geq 1)$.
By Proposition~\ref{basisR6}, $\Sigma$ is an infinite equational basis of $\mathsf{V}(TR_6)$.
By Lemma~\ref{compactness}, it suffices to prove that no finite subset of $\Sigma$ defines $\mathsf{V}(TR_6)$.
Let $\Sigma'$ be an arbitrary finite subset of $\Sigma$. Choose an integer $n\geq2$ greater than every index $k$ for which $\sigma_k\in\Sigma'$.
Let $\Sigma_n$ consist of the identities~\eqref{eq2}, \eqref{eq3}, \eqref{eq1}, and $\sigma_k$ for $1\leq k<n$.
Then $\Sigma_n$ contains $\Sigma'$, but does not contain $\sigma_n$.
To show that $\Sigma'$ cannot define $\mathsf{V}(TR_6)$,
it remains to prove that $\Sigma_n$ cannot derive the inequality $\sigma_n$.

Suppose, for contradiction, that $\Sigma_n$ can derive the inequality $\sigma_n$.
By Lemma~\ref{eqyield},
there exist terms $\mathbf{t}_1, \mathbf{t}_2, \ldots, \mathbf{t}_m \in P_f(X_c^+)$ such that $\mathbf{u}^{(n)}=\mathbf{t}_1$, $\mathbf{u}^{(n)}+x=\mathbf{t}_m$ and, for each $1\leq i<m$, there are terms $\mathbf{p}_i,\mathbf{r}_i,\mathbf{s}_i,\mathbf{s}_i'\in P_f(X_c^+)$ and a substitution $\varphi_i\colon P_f(X_c^+) \to P_f(X_c^+)$ such that
\[
\mathbf{t}_i=\mathbf{p}_i\varphi_i(\mathbf{s}_i)+\mathbf{r}_i,\quad
\mathbf{t}_{i+1}=\mathbf{p}_i\varphi_i(\mathbf{s}_i')+\mathbf{r}_i,
\]
where $\mathbf{s}_i\approx\mathbf{s}_i' \in \Sigma_n$ or $\mathbf{s}_i'\approx\mathbf{s}_i \in \Sigma_n$,
$\mathbf{p}_i$ may be the empty word, and $\mathbf{r}_i$ may be the empty set.

We shall show by induction on $i$ that $\mathbf{t}_i = \mathbf{u}^{(n)} + \mathbf{w}_i$ for every $2 \leq i \leq m$,
where $\mathbf{w}_i$ is a (possibly empty) sum of variables in $c(\mathbf{u}^{(n)})$.
Thus $\mathbf{t}_m = \mathbf{u}^{(n)} + \mathbf{w}_m$,
where $\mathbf{w}_m$ is a sum of some variables in $c(\mathbf{u}^{(n)})$.
This contradicts the fact that $\mathbf{t}_m = \mathbf{u}^{(n)} + x$, since $x \notin c(\mathbf{u}^{(n)})$.
Hence $\Sigma_n$ cannot derive $\sigma_n$, and so $\Sigma'$ cannot define $\mathsf{V}(TR_6)$.
Therefore, $TR_6$ is nonfinitely based.

Indeed, if $i=2$, then $\mathbf{t}_{2}=\mathbf{p}_1\varphi_1(\mathbf{s}_1')+\mathbf{r}_1$ and $\mathbf{p}_1\varphi_1(\mathbf{s}_1)+\mathbf{r}_1=\mathbf{t}_{1}=\mathbf{u}^{(n)}$,
where $\mathbf{s}_1\approx\mathbf{s}_1' \in \Sigma_n$ or $\mathbf{s}_1'\approx\mathbf{s}_1 \in \Sigma_n$.
Observe that $\mathbf{u}^{(n)}$ is $x^2$-free, $x_1x_2x_3$-free, and $(xy+x)$-free.
Also, by \cite[Proposition~3.1]{LyuRenYue2026}, $\mathbf{u}^{(n)}$ is $\mathbf{u}^{(k)}$-free for each $1\leq k<n$.
Thus $\mathbf{s}_1=xy$ and $\mathbf{s}_1'=xy+x$, and so $\mathbf{p}_1$ is empty.
Consequently,
\begin{equation}\label{26082701}
\mathbf{u}^{(n)}=\varphi_1(xy)+\mathbf{r_1}.
\end{equation}
Therefore, $\varphi_1(xy)=x_ix_{i+1}$ or $x_{i+1}(x_i+x_{i+2})$ for some $i$,
where indices are taken modulo $2n+1$.
If $\varphi_1(xy)=x_ix_{i+1}$, then $\{\varphi_1(x),\varphi_1(y)\}=\{x_i,x_{i+1}\}$.
If $\varphi_1(xy)=x_{i+1}(x_i+x_{i+2})$, then $\{\varphi_1(x), \varphi_1(y)\}=\{x_i+x_{i+2},x_{i+1}\}$.
Thus $\varphi_1(x)$ is either a single variable or the sum of two variables, all belonging to $c(\mathbf{u}^{(n)})$.
Hence
\[
\begin{aligned}
\mathbf{t}_{2}
&=\mathbf{p}_1\varphi_1(\mathbf{s}_1')+\mathbf{r}_1
 =\varphi_1(xy+x)+\mathbf{r}_1 \\
&=\varphi_1(xy)+\varphi_1(x)+\mathbf{r}_1
 =\varphi_1(xy)+\mathbf{r}_1+\varphi_1(x)
 \stackrel{\eqref{26082701}}=\mathbf{u}^{(n)}+\varphi_1(x).
\end{aligned}
\]
This proves the base case for the induction.

Now let $2 \leq i < m$.
Suppose that $\mathbf{t}_i = \mathbf{u}^{(n)} + \mathbf{w}_i$,
where $\mathbf{w}_i$ is a sum of variables in $c(\mathbf{u}^{(n)})$.
Then $\mathbf{u}^{(n)} + \mathbf{w}_i=\mathbf{p}_i\varphi_i(\mathbf{s}_i)+\mathbf{r}_i$
and $\mathbf{p}_i\varphi_i(\mathbf{s}_i')+\mathbf{r}_i=\mathbf{t}_{i+1}$,
where $\mathbf{s}_i\approx\mathbf{s}_i' \in \Sigma_n$ or $\mathbf{s}_i'\approx\mathbf{s}_i \in \Sigma_n$.
Since $\mathbf{u}^{(n)}$ is $x^2$-free, $x_1x_2x_3$-free,
and $\mathbf{u}^{(n)}$ is $\mathbf{u}^{(k)}$-free for each $1\leq k<n$,
it follows that either $\mathbf{s}_i=xy$ and $\mathbf{s}_i'=xy+x$, or $\mathbf{s}_i=xy+x$ and $\mathbf{s}_i'=xy$,
and hence $\mathbf{p}_i$ is empty.

If $\mathbf{s}_i=xy$ and $\mathbf{s}_i'=xy+x$, then
\[
\mathbf{u}^{(n)} + \mathbf{w}_i=\varphi_i(\mathbf{s}_i)+\mathbf{r}_i=\varphi_i(xy)+\mathbf{r}_i,
\]
and so
\begin{equation}\label{26082702}
\mathbf{u}^{(n)} + \mathbf{w}_i=\varphi_i(xy)+\mathbf{r}_i.
\end{equation}
This implies that $\varphi_i(x)$ is either a single variable or the sum of two variables, all belonging to $c(\mathbf{u}^{(n)})$.
Thus
\[
\begin{aligned}
\mathbf{t}_{i+1}
&=\varphi_i(\mathbf{s}_i')+\mathbf{r}_i
 =\varphi_i(xy+x)+\mathbf{r}_i =\varphi_i(xy)+\varphi_i(x)+\mathbf{r}_i\\
&=\varphi_i(xy)+\mathbf{r}_i+\varphi_i(x)
  \stackrel{\eqref{26082702}}=\mathbf{u}^{(n)} + \mathbf{w}_i+\varphi_i(x)
 =\mathbf{u}^{(n)} + \mathbf{w}_{i+1},
\end{aligned}
\]
where $\mathbf{w}_{i+1}=\mathbf{w}_i+\varphi_i(x)$, which is a sum of variables in $c(\mathbf{u}^{(n)})$.

If $\mathbf{s}_i=xy+x$ and $\mathbf{s}_i'=xy$, then
\[
\mathbf{u}^{(n)} + \mathbf{w}_i=\varphi_i(\mathbf{s}_i)+\mathbf{r}_i=\varphi_i(xy+x)+\mathbf{r}_i
=\varphi_i(xy)+\varphi_i(x)+\mathbf{r}_i,
\]
and so $\varphi_i(x)$ is either a single variable or the sum of two variables, all belonging to $c(\mathbf{u}^{(n)})$.
Thus $\varphi_i(x)$ is an additive subterm of $\mathbf{w}_i$,
and so $\mathbf{u}^{(n)}$ is an additive subterm of $\varphi_i(xy)+\mathbf{r}_i$.
Consequently, $\mathbf{r}_i=\mathbf{r}_{i1}+\mathbf{r}_{i2}$,
where $\mathbf{r}_{i1}$ is an additive subterm of $\mathbf{u}^{(n)}$
and $\mathbf{r}_{i2}$ is an additive subterm of $\mathbf{w}_i$.
Note that $\varphi_i(xy)$ is an additive subterm of $\mathbf{u}^{(n)}$.
Then \begin{equation}\label{26082703}
\varphi_i(xy) + \mathbf{r}_{i1} = \mathbf{u}^{(n)}.
\end{equation}
Thus
\[
\mathbf{t}_{i+1}
= \varphi_i(\mathbf{s}_i') + \mathbf{r}_i
= \varphi_i(xy) + \mathbf{r}_i
= \varphi_i(xy) + \mathbf{r}_{i1} + \mathbf{r}_{i2}
\stackrel{\eqref{26082703}}= \mathbf{u}^{(n)} + \mathbf{w}_{i+1},
\]
where $\mathbf{w}_{i+1}=\mathbf{r}_{i2}$, which is a sum of variables in $c(\mathbf{u}^{(n)})$.
This completes the proof.
\end{proof}

\section{The subvariety lattice of the variety $\mathsf{V}(TR_6)$}
In the previous section, we proved that $\mathsf{V}(TR_6)$ is a nonfinitely based variety.
In this section, we characterize the subvariety lattice of $\mathsf{V}(TR_6)$,
and show that it forms a four-element chain.
Moreover, every proper subvariety of $\mathsf{V}(TR_6)$ is finitely based.
Consequently, $\mathsf{V}(TR_6)$ is a limit variety.

We begin by determining the minimal nontrivial subvarieties of $\mathsf{V}(TR_6)$.
Let $T_2$ denote the two-element ai-semiring $\{0, 1\}$ whose Cayley tables are given in Table~\ref{T2}.
Then $T_2$ is isomorphic to the subalgebra $\{1, 3\}$ of $TR_6$.
Thus $T_2$ lies in $\mathsf{V}(TR_6)$.

\begin{table}[htbp]
\centering
\small
\caption{The Cayley tables of $T_2$}\label{T2}
\renewcommand{\arraystretch}{1.08}
\setlength{\tabcolsep}{4.8pt}
\begin{tabular}{c|cc}
$+$ & 0 & 1  \\ \hline
0 & 0 & 1 \\
1 & 1 & 1
\end{tabular}
\qquad\qquad
\begin{tabular}{c|cc}
$\cdot$ & 0 & 1  \\ \hline
0 & 1 & 1 \\
1 & 1 & 1
\end{tabular}
\end{table}

\begin{proposition}\label{minimal-subvariety}
The variety $\mathsf{V}(T_2)$ is the only minimal nontrivial subvariety of $\mathsf{V}(TR_6)$.
\end{proposition}
\begin{proof}
Since $T_2$ lies in $\mathsf{V}(TR_6)$, the variety $\mathsf{V}(T_2)$ is a subvariety of $\mathsf{V}(TR_6)$.
By \cite[Theorem~1.1]{Polin1980} and \cite[Table~1]{ShaoRen2015},
the minimal nontrivial subvarieties of the variety of all ai-semirings are precisely
\[
\mathsf{V}(L_2),\quad \mathsf{V}(R_2),\quad \mathsf{V}(M_2),\quad
\mathsf{V}(D_2),\quad \mathsf{V}(N_2),\quad \mathsf{V}(T_2).
\]
Observe that none of $L_2, R_2, M_2, D_2$, or $N_2$ satisfies identity~\eqref{eq2}.
Therefore, $\mathsf{V}(T_2)$ is the only minimal nontrivial subvariety of $\mathsf{V}(TR_6)$.
\end{proof}

By \cite[Table~1]{ShaoRen2015}, $T_2$ is finitely based.
More precisely, $\mathsf{V}(T_2)$ is the ai-semiring variety defined by the identities
\begin{align}
&x \preceq x^2; \label{eq50}\\
&x_1x_2 \approx y_1y_2. \label{eq51}
\end{align}
We now determine an equational basis for $\mathsf{V}(T_2)$ within $\mathsf{V}(TR_6)$.

\begin{proposition}\label{Sa-subvariety}
The variety $\mathsf{V}(T_2)$ is the subvariety of $\mathsf{V}(TR_6)$ defined by the identity
\begin{equation}
x^2\approx xy.\label{eq7}
\end{equation}
\end{proposition}
\begin{proof}
It is easy to see that $T_2$ satisfies the identity \eqref{eq7}.
It remains to show that \eqref{eq50} and \eqref{eq51} are derivable from \eqref{eq2} and \eqref{eq7}.
First, \eqref{eq50} follows directly from \eqref{eq2} by taking $y=x$.
Second, using \eqref{eq2} and \eqref{eq7}, we have
\[
x_1x_2
\overset{\eqref{eq7}}{\approx} x_1^2
\overset{\eqref{eq2}}{\approx} x_1^2 + y_1^2
\approx y_1^2 + x_1^2
\overset{\eqref{eq2}}{\approx} y_1^2
\overset{\eqref{eq7}}{\approx} y_1y_2.
\]
Thus \eqref{eq51} is also derivable.
Therefore, $\mathsf{V}(T_2)$ is precisely the subvariety of $\mathsf{V}(TR_6)$ defined by \eqref{eq7}.
\end{proof}

\begin{table}[htbp]
\centering
\small
\caption{The Cayley tables of \(S_{(4, 369)}\)}\label{S4-369}
\renewcommand{\arraystretch}{1.08}
\setlength{\tabcolsep}{4.8pt}
\begin{tabular}{c|cccc}
$+$ & 1 & 2 & 3 & 4 \\ \hline
1 & 1 & 2 & 1 & 1 \\
2 & 2 & 2 & 2 & 2 \\
3 & 1 & 2 & 3 & 1 \\
4 & 1 & 2 & 1 & 4
\end{tabular}
\qquad\qquad
\begin{tabular}{c|cccc}
$\cdot$ & 1 & 2 & 3 & 4 \\ \hline
1 & 2 & 2 & 2 & 2 \\
2 & 2 & 2 & 2 & 2 \\
3 & 2 & 2 & 2 & 1 \\
4 & 2 & 2 & 1 & 2
\end{tabular}
\end{table}

Let $S_{(4, 369)}$ denote the four-element ai-semiring $\{1, 2, 3, 4\}$ whose Cayley tables are given in Table~\ref{S4-369}.
Then $S_{(4, 369)}$ is isomorphic to the subalgebra $\{1, 3, 5, 6\}$ of $TR_6$.
Hence $S_{(4, 369)}$ lies in the variety $\mathsf{V}(TR_6)$.
The following result, due to Ren et al.~\cite[Proposition 6.8]{RenLiuYueChen2026}, shows that $S_{(4, 369)}$ is finitely based.

\begin{lemma}\label{lemma26082720}
The commutative ai-semiring variety $\mathsf{V}(S_{(4, 369)})$ is defined by the identities \eqref{eq2}, \eqref{eq3}, together with
\begin{align}
&x_4\preceq x_1x_2x_3; \label{eq30}\\
&x_1x_4\preceq x_1x_2+x_2x_3+x_3x_4. \label{eq31}
\end{align}
\end{lemma}
%\begin{proof}
%This is due to Ren et al.~\cite[Proposition 6.8]{RenLiuYueChen2026}.
%\end{proof}

The next proposition determines an equational basis for $\mathsf{V}(S_{(4, 369)})$ within $\mathsf{V}(TR_6)$.

\begin{proposition}\label{Sstar-subvariety}
The variety $\mathsf{V}(S_{(4, 369)})$ is the subvariety of $\mathsf{V}(TR_6)$ defined by the inequality~\eqref{eq31}.
\end{proposition}
\begin{proof}
Since $S_{(4, 369)}$ lies in $\mathsf{V}(TR_6)$, it follows that $\mathsf{V}(S_{(4, 369)})$ is a subvariety of $\mathsf{V}(TR_6)$.
By Lemma~\ref{lemma26082720}, the commutative ai-semiring variety
$\mathsf{V}(S_{(4, 369)})$ is defined by the identities \eqref{eq2}, \eqref{eq3}, \eqref{eq30}, and \eqref{eq31}.
By Lemma~\ref{R6identities}, $TR_6$ satisfies \eqref{eq2}, \eqref{eq3}, and \eqref{eq1}.
Moreover, \eqref{eq30} is derivable from \eqref{eq2} and \eqref{eq1}, so $TR_6$ also satisfies \eqref{eq30}.
Therefore, $\mathsf{V}(S_{(4, 369)})$ is the subvariety of $\mathsf{V}(TR_6)$ defined by \eqref{eq31}.
\end{proof}

With the above preparations in hand, we now describe the subvariety lattice of $\mathsf{V}(TR_6)$.

\begin{proposition}\label{subvariety-lattice}
The subvariety lattice of $\mathsf{V}(TR_6)$ is a four-element chain (see Figure~\ref{lattice-figure}),
where $\mathbf T$ denotes the trivial variety.

\begin{figure}[ht]
\centering
\begin{tikzpicture}[
    x=1cm, y=1.2cm,
    lattice point/.style={
        circle, minimum size=5pt, inner sep=0pt,
        fill=white, draw=black!90, line width=0.8pt,
    }
]

% 先画线
\draw[line width=0.7pt, draw=black!60] (0,0) -- (0,1) -- (0,2) -- (0,3);

% 再画点（压住线）
\fill[lattice point] (0,0) circle (2pt) node[left=10pt] {$\mathbf T$};
\fill[lattice point] (0,1) circle (2pt) node[left=10pt] {$\mathsf{V}(T_2)$};
\fill[lattice point] (0,2) circle (2pt) node[left=10pt] {$\mathsf{V}(S_{(4, 369)})$};
\fill[lattice point] (0,3) circle (2pt) node[left=10pt] {$\mathsf{V}(TR_6)$};

% 高光点
\foreach \y in {0,1,2,3} {
    \fill[white, opacity=0.8] (0.12,\y+0.12) circle (1.2pt);
}

\end{tikzpicture}
\caption{The subvariety lattice of $\mathsf{V}(TR_6)$.}
\label{lattice-figure}
\end{figure}
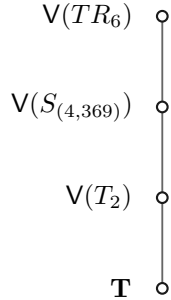
\end{proposition}
\begin{proof}
Let $\mathcal V$ be a nontrivial subvariety of $\mathsf{V}(TR_6)$ distinct from $\mathsf{V}(T_2)$.
By Proposition~\ref{minimal-subvariety}, $\mathcal V$ properly contains $\mathsf{V}(T_2)$.
It follows from Proposition~\ref{Sa-subvariety} that $\mathcal V$ does not satisfy identity~\eqref{eq7}.
Hence there exists an ai-semiring $A\in\mathcal V$ with elements $a,b\in A$ such that $a^2\ne ab$, and so $a\ne b$.
We shall establish the following relations:
\[
ab<a^2,\qquad a<ab,\qquad b<ab,\qquad a+b\leq ab,
\]
\[
a+b < a^2,\qquad b \nleq a,\qquad a \nleq b.
\]
First, by the identity~\eqref{eq2}, we obtain $ab\leq a^2$.
Since $a^2\ne ab$, it follows that $ab<a^2$.
The identity~\eqref{eq3} gives $a\leq ab$ and $b\leq ab$. If $ab=a$, then
\[
ab=ab^2\overset{\eqref{eq1}}{=}a^2,
\]
a contradiction; hence $a<ab$. A symmetric argument shows that $b<ab$. Consequently, $a+b\leq ab$.
Combining this with $ab<a^2$, we have $a+b< a^2$.

By \eqref{eq2}, one can derive the identity
\begin{equation}\label{26082825}
x^2 \approx y^2.
\end{equation}
If $b \leq a$, then $b+a=a$, and so
\[
ab=(b+a)b=b^2+ab
\overset{\eqref{eq2}}{=}b^2\overset{\eqref{26082825}}{=}a^2,
\]
a contradiction; thus $b \nleq a$. Similarly, $a \nleq b$.

Let $\langle a,b\rangle$ denote the subalgebra of $A$ generated by $a$ and $b$.
From the relations established above, together with the identities~\eqref{eq2}, \eqref{eq3}, and \eqref{eq1},
it follows that $\langle a, b\rangle=\{a, b, a+b, ab, a^2\}$,
where $ab$ and $a+b$ may possibly coincide, but all other pairs among these five elements are distinct.
Set $I=\{ab,a+b\}$. It is readily verified that the equivalence relation $\rho$ on $\langle a, b\rangle$
whose only possible nonsingleton class is $I$ is a semiring congruence.
The quotient algebra $\langle a, b\rangle / \rho$ is easily seen to be isomorphic to $S_{(4, 369)}$
under the mapping $\{a\} \mapsto 3$, $\{b\} \mapsto 4$, $I \mapsto 1$, and $\{a^2\} \mapsto 2$.
Consequently, $S_{(4, 369)}$ lies in $\mathsf{V}(A)$, and therefore $\mathsf{V}(S_{(4, 369)})$ is a subvariety of $\mathcal V$.

If $\mathcal V\ne\mathsf{V}(S_{(4, 369)})$,
then by Proposition~\ref{Sstar-subvariety}, $\mathcal V$ does not satisfy the inequality~\eqref{eq31}.
Hence there exists an ai-semiring $B\in\mathcal V$ with elements $a,b,c,d\in B$ such that
\[
ad\nleq ab+bc+cd.
\]
Our aim is to show that $\mathsf{V}(B)$ contains $TR_6$, which then implies that $\mathcal V=\mathsf{V}(TR_6)$.

Let $\langle a,b,c,d\rangle$ denote the subalgebra of $B$ generated by $a,b,c,d$.
Then by the identities~\eqref{eq2}, \eqref{eq3}, \eqref{eq1}, and $\sigma_n$ for all $n \geq 1$,
every element of $\langle a, b, c, d\rangle$ is a finite sum of elements from the set
\[
\{a, b, c, d, ab, ac, ad, bc, bd, cd, a^2\}.
\]
In particular, $a^2$ acts as both the additive maximum element and the multiplicative zero.
One can verify that $\langle a,b,c,d\rangle$ contains at most $86$ elements.
Consider the following subsets of $\langle a,b,c,d\rangle$:
\[
\begin{aligned}
R_2 &= \{a,\; a+c\}, \\
R_3 &= \bigl\{ab,\; ab+c,\; ab+d,\; ab+c+d,\; bc,\; bc+a,\; bc+d,\; bc+a+d, \\
     &\qquad cd,\; cd+a,\; cd+b,\; cd+a+b,\; ab+bc,\; ab+bc+d,\; bc+cd,\; bc+cd+a, \\
     &\qquad ab+cd,\; ab+bc+cd,\; a+b,\; a+d,\; a+b+d,\; b+c,\; c+d,\; b+c+d, \\
     &\qquad a+b+c,\; a+c+d,\; a+b+c+d \bigr\}, \\
R_4 &= \{d,\; b+d\}, \\
R_5 &= \{c\}, \\
R_6 &= \{b\}, \\
R_1 &= \langle a,b,c,d\rangle \setminus (R_2 \cup R_3 \cup R_4 \cup R_5 \cup R_6).
\end{aligned}
\]

A direct verification (which is straightforward but somewhat lengthy) confirms that the subsets $R_1,\ldots,R_6$ are pairwise disjoint and that the equivalence relation $\tau$ with classes $R_1,\ldots,R_6$ is compatible with addition and multiplication;
we omit the details for brevity. Consequently, $\tau$ is a semiring congruence on $\langle a,b,c,d\rangle$.

Finally, one can show that the quotient algebra $\langle a, b, c, d \rangle/\tau$
%\[
%\langle a, b, c, d \rangle/\rho=\{R_i \mid 1\leq i\leq6\}
%\]
is isomorphic to $TR_6$,
where the isomorphism sends each congruence class $R_i$ to the element $i$ of $TR_6$, $1\leq i\leq6$.
Thus $\mathsf{V}(B)$ contains $TR_6$, and so $\mathcal V=\mathsf{V}(TR_6)$.
\end{proof}

We now arrive at the main result of this section.

\begin{theorem}\label{limit-variety}
The variety $\mathsf{V}(TR_6)$ is a limit variety.
\end{theorem}
\begin{proof}
By Propositions~\ref{nfb} and~\ref{subvariety-lattice}, together with Lemma~\ref{lemma26082720},
$\mathsf{V}(TR_6)$ itself is nonfinitely based, while all of its proper subvarieties are finitely based.
Therefore, $\mathsf{V}(TR_6)$ is a limit variety.
\end{proof}

\section{The variety $\mathsf{V}(SR_6, TR_6)$}
Based on the results of the previous sections, it is natural to study a larger variety,
namely the variety $\mathsf{V}(SR_6, TR_6)$ generated by $SR_6$ and $TR_6$.
In this section, we first characterize the inequalities satisfied by
$\mathsf{V}(SR_6, TR_6)$ and then show that this variety is also nonfinitely based.
We conclude the section by introducing two important members of $\mathsf{V}(SR_6, TR_6)$
and establishing some of their properties.

We begin by showing that the variety $\mathsf{V}(SR_6, TR_6)$ can in fact be generated by a single seven-element ai-semiring.
Let $ST_7$ denote the seven-element ai-semiring whose Cayley tables are given in Table~\ref{tab:ST_7}.

\begin{table}[htbp]
\centering
\caption{The Cayley tables of $ST_7$}\label{tab:ST_7}
\small
\renewcommand{\arraystretch}{1.08}
\setlength{\tabcolsep}{4.8pt}
\begin{tabular}{c|ccccccc}
$+$&1&2&3&4&5&6&7\\ \hline
1&1&1&1&1&1&1&1\\
2&1&2&7&7&2&7&7\\
3&1&7&3&7&7&7&7\\
4&1&7&7&4&7&4&7\\
5&1&2&7&7&5&7&7\\
6&1&7&7&4&7&6&7\\
7&1&7&7&7&7&7&7
\end{tabular}
\hspace{1.2cm}
\begin{tabular}{c|ccccccc}
$\cdot$&1&2&3&4&5&6&7\\ \hline
1&1&1&1&1&1&1&1\\
2&1&1&1&1&1&3&1\\
3&1&1&1&1&1&1&1\\
4&1&1&1&1&3&1&1\\
5&1&1&1&3&1&3&1\\
6&1&3&1&1&3&1&1\\
7&1&1&1&1&1&1&1
\end{tabular}
\end{table}

\begin{proposition}\label{prop:C7-join}
$\mathsf{V}(ST_7)=\mathsf{V}(SR_6, TR_6)$.
\end{proposition}
\begin{proof}
It is straightforward to check that $ST_7$ is isomorphic to a subdirect product
of $SR_6$ and $TR_6$ via the congruences defined by the nontrivial blocks
$\{1, 7\}$ and $\{3, 7\}$, respectively.
Hence $\mathsf{V}(ST_7)=\mathsf{V}(SR_6, TR_6)$.
\end{proof}

In view of the equality in Proposition~\ref{prop:C7-join}, we shall henceforth study $\mathsf{V}(ST_7)$.
The following result presents an infinite equational basis for $ST_7$.

\begin{proposition}\label{prop:basis-C7}
The commutative ai-semiring variety $\mathsf{V}(ST_7)$ is defined by the identities
$\sigma_n$ $(n\geq1)$, \eqref{eq2}, \eqref{eq1}, and
\begin{equation}\label{eq:Cbridge}
x_1 \preceq x_2+x_2x_3+x_1x_4.
\end{equation}
\end{proposition}
\begin{proof}
It is routine to check that $ST_7$ satisfies the identities~\eqref{eq2}, \eqref{eq1}, and \eqref{eq:Cbridge}.
By Lemma~\ref{R6identities} and Remark~\ref{remark26082910},
both $SR_6$ and $TR_6$ satisfy the inequality $\sigma_n$ for all $n \geq 1$.
It follows from Proposition~\ref{prop:C7-join} that $ST_7$ also satisfies the inequality $\sigma_n$ for all $n \geq 1$.

It remains to show that every inequality satisfied by $ST_7$ is derivable from
the identities \eqref{eq2}, \eqref{eq1}, \eqref{eq:Cbridge}, and $\sigma_n$ $(n\geq1)$.
Let $\mathbf q\preceq\mathbf u$ be such a
nontrivial inequality, where $\bu=\bu_1+\bu_2+\cdots+\bu_n$ with $\bu_i, \bq \in X_c^+$ for all $1\leq i \leq n$.
Since $TR_6$ is a homomorphic image of $ST_7$,
it follows that $TR_6$ also satisfies $\mathbf q\preceq\mathbf u$.
By Proposition~\ref{wordprop}, the following four cases arise:
\begin{enumerate}[(\rm i)]
\item $\ell(\mathbf u_i) \geq 3$ for some $\mathbf u_i \in \mathbf u$;
\item $\mathbf u$ contains a non-linear word;
\item the graph $\mathbb{G}_{\mathbf u}$ contains an odd cycle;
\item $\ell(\mathbf q)=1$ and $c(\mathbf q)\subseteq c(\mathbf u)$.
\end{enumerate}
The proofs for the first three cases are identical to those in the proof of Proposition~\ref{basisR6}
(indeed, they rely only on the identities $\sigma_n$ $(n \geq 1)$, \eqref{eq2}, and \eqref{eq1}),
and are therefore omitted.

It remains to consider the fourth case.
Now suppose that $\ell(\mathbf q)=1$ and $c(\mathbf q)\subseteq c(\mathbf u)$. We write $\bq=z$.
Since Cases 1--3 are already excluded, we may also assume that
$\ell(\mathbf u_i) \leq 2$ for all $\mathbf u_i \in \mathbf u$,
$\bu_i$ is a linear word for all $\bu_i\in \bu$, and $\mathbb{G}_{\mathbf u}$ contains no odd cycle.
Then $\bu=L_1(\bu)+L_2(\bu)$, $zt\in \bu$ for some $t \in X$, and $\mathbb G_{\mathbf u}$ is bipartite.
So we can write $L_2(\mathbf{u})=x_1y_1+x_2y_2+\cdots+x_my_m$,
where $\{x_1, x_2, \ldots, x_m\}$ and $\{y_1, y_2, \ldots, y_m\}$ are disjoint.
Set $A=\{x_1, x_2, \ldots, x_m\}$ and $B=\{y_1, y_2, \ldots, y_m\}$.
We claim that $c(L_1(\mathbf u))\cap c(L_2(\mathbf u)) \neq \emptyset$.
Suppose that this is not true.
Define a semiring homomorphism $\varphi:P_f(X_c^+)\to ST_7$ by
\[
\varphi(v)=
\begin{cases}
5,&v\in A,\\
6,&v\in B,\\
3,&v\in c(L_1(\mathbf u)),\\
1,&\text{otherwise}.
\end{cases}
\]
Since $c(L_1(\mathbf u))\cap c(L_2(\mathbf u)) = \emptyset$,
it follows that $\varphi$ is well-defined.
Now we have
\[
\varphi(\mathbf u)=\varphi(L_1(\bu)+L_2(\bu))=\varphi(L_1(\bu))+\varphi(L_2(\bu))=3+5\cdot 6=3+3=3,
\]
whereas $\varphi(\bq)=\varphi(z)=5$ or $6$.
Since $5 \nleq 3$ and $6 \nleq 3$ in $ST_7$, it follows that $\varphi(\bq) \nleq \varphi(\mathbf u)$,
which contradicts the fact that $\bq \preceq \bu$ holds in $ST_7$.
Thus $c(L_1(\mathbf u))\cap c(L_2(\mathbf u)) \neq \emptyset$,
and so $x \in c(L_1(\mathbf u))\cap c(L_2(\mathbf u))$ for some $x\in X$.
Hence $x+xy$ is an additive subterm of $\bu$ for some $y \in X$, and so
\[
\mathbf u\succeq x+xy+zt
\overset{\eqref{eq:Cbridge}}{\succeq}z=\bq.
\]
This derives the inequality $\bu \succeq \bq$.
\end{proof}

\begin{remark}
We shall use the following facts without further comment.
By Proposition~\ref{prop:basis-C7}, in any ai-semiring $S$ belonging to $\mathsf{V}(ST_7)$,
the square of any element, as well as the product of three or more elements,
is both the additive maximum element and the multiplicative zero element.
\end{remark}

We now apply Proposition~\ref{prop:basis-C7} to determine the finite basis status of $ST_7$.

\begin{corollary}\label{cor:C7-nfb}
The ai-semiring $ST_7$ is nonfinitely based.
\end{corollary}
\begin{proof}
Let $\Sigma$ denote the set $\{\eqref{eq2},\eqref{eq1},\eqref{eq:Cbridge}\}\cup\{\sigma_n\mid n\geq1\}$.
By Proposition~\ref{prop:basis-C7}, $\Sigma$ is an equational basis of $ST_7$.
By Lemma~\ref{compactness}, it suffices to show that no finite subset of $\Sigma$ defines
$\mathsf{V}(ST_7)$.

Let $\Sigma'$ be an arbitrary finite subset of $\Sigma$. Choose $n\geq2$ greater
than every index $k$ for which $\sigma_k\in\Sigma'$, and put
\[
\Sigma_n=\{\eqref{eq2},\eqref{eq1},\eqref{eq:Cbridge}\}
\cup\{\sigma_k\mid1\leq k<n\}.
\]
Then $\Sigma'$ is a subset of $\Sigma_n$ and $\sigma_n$ is not in $\Sigma_n$.
To prove that $\Sigma'$ cannot define $\mathsf{V}(ST_7)$,
we show that $\Sigma_n$ cannot derive $\sigma_n$.

Indeed, observe that the term $\mathbf{u}^{(n)}$ is $x^2$-free,
$x_1x_2x_3$-free, and $(x+xy)$-free. Moreover, by
\cite[Proposition~3.1]{LyuRenYue2026}, $\mathbf{u}^{(n)}$ is
$\mathbf{u}^{(k)}$-free for every $1\leq k<n$.
Consequently, for any identity in $\Sigma_n$ and any term $\bs$ occurring on either side of it,
there do not exist terms $\bp, \br$ and a substitution $\varphi$ such that
\[
\mathbf{u}^{(n)} = \bp \varphi(\bs) + \br.
\]
By Lemma~\ref{eqyield}, $\Sigma_n$ cannot derive $\sigma_n$.
Therefore, $ST_7$ is nonfinitely based.
\end{proof}

\begin{corollary}\label{prop:define-SR6-ST7}
The variety $\mathsf{V}(SR_6)$ is the subvariety of $\mathsf{V}(ST_7)$
defined by the identity
\begin{equation}\label{id2608090150}
x^2\approx x+xy.
\end{equation}
\end{corollary}
\begin{proof}
By \cite[Proposition~4.5]{LyuRenYue2026},
the commutative ai-semiring variety $\mathsf{V}(SR_6)$ is defined by the identity
\eqref{id2608090150}, together with the identities
\begin{align}
&x^3\approx x^2; \label{id2608090151}\\
&x_1\preceq x_2x_3x_4; \label{id2608090152}\\
& x_1x_2\cdots x_{2n+1}\preceq \bu^{(n)} \quad (n\geq 1). \label{id2608090153}
\end{align}
By Proposition~\ref{prop:basis-C7},
it remains to show that \eqref{id2608090151}, \eqref{id2608090152}, and \eqref{id2608090153}
are derivable from \eqref{eq2}, \eqref{eq1}, \eqref{eq:Cbridge}, and $\sigma_n$ for all $n \geq 1$.
Indeed, it is easy to see that \eqref{eq1} implies \eqref{id2608090151}.
Next, using \eqref{eq2} and \eqref{eq1}, we have
\[
x_1\overset{\eqref{eq2}}{\preceq}x_2^2
\overset{\eqref{eq1}}{\approx}x_2x_3x_4.
\]
This proves \eqref{id2608090152}.
Finally, \eqref{id2608090153} follows directly from $\sigma_n$ $(n \geq 1)$.
This completes the proof.
\end{proof}

\begin{corollary}\label{prop:define-TR6-ST7}
The variety $\mathsf{V}(TR_6)$ is the subvariety of $\mathsf{V}(ST_7)$
defined by the identity~\eqref{eq3}.
\end{corollary}
\begin{proof}
This follows immediately from Propositions \ref{basisR6} and \ref{prop:basis-C7}.
\end{proof}

\begin{corollary}\label{prop:define-T2-ST7}
The variety $\mathsf{V}(T_2)$ is the subvariety of $\mathsf{V}(ST_7)$
defined by identity~\eqref{eq7}.
\end{corollary}
\begin{proof}
By Proposition~\ref{Sa-subvariety}, $\mathsf{V}(T_2)$ is the subvariety of $\mathsf{V}(TR_6)$ defined by \eqref{eq7}.
By Corollary~\ref{prop:define-TR6-ST7}, $\mathsf{V}(TR_6)$ is the subvariety of $\mathsf{V}(ST_7)$
defined by the identity~\eqref{eq3}.
Hence $\mathsf{V}(T_2)$ is the subvariety of $\mathsf{V}(ST_7)$ defined by the identities \eqref{eq3} and \eqref{eq7}.
To prove the required inclusion, it suffices to show that \eqref{eq3} is derivable from \eqref{eq7} and \eqref{eq2}
(the latter holds in $ST_7$).
Indeed, we have
\[
x \overset{\eqref{eq2}}\preceq y^2 \overset{\eqref{eq7}}\approx yx \approx xy.
\]
This derives the inequality \eqref{eq3}.
\end{proof}

To apply Corollaries~\ref{prop:define-SR6-ST7}--\ref{prop:define-T2-ST7} to determine the intersection of $\mathsf{V}(SR_6)$ and $\mathsf{V}(TR_6)$,
we recall the following result of Lyu et al.~\cite[Proposition 5.3, Theorem 5.4]{LyuRenYue2026},
which describes the subvarieties of $\mathsf{V}(SR_6)$.

\begin{lemma}\label{lem26090180}
The variety $\mathsf{V}(SR_6)$ is a limit variety and
has exactly four subvarieties, ordered by inclusion: $\mathsf{V}(SR_6)$,
$\mathsf{V}(S_c(ab))$, $\mathsf{V}(T_2)$, and $\mathbf{T}$,
where the Cayley tables of $S_c(ab)$ are given in Table~$\ref{tab:S48}$.
\end{lemma}

\begin{table}[htbp]
\centering
\small
\caption{The Cayley tables of $S_c(ab)$}\label{tab:S48}
\renewcommand{\arraystretch}{1.08}
\setlength{\tabcolsep}{4.8pt}
\begin{tabular}{c|cccc}
$+$&0&$a$&$b$&$ab$\\ \hline
0&0&0&0&0\\
$a$&0&$a$&0&0\\
$b$&0&0&$b$&0\\
$ab$&0&0&0&$ab$
\end{tabular}
\qquad\qquad
\begin{tabular}{c|cccc}
$\cdot$&0&$a$&$b$&$ab$\\ \hline
0&0&0&0&0\\
$a$&0&$0$&$ab$&0\\
$b$&0&$ab$&$0$&0\\
$ab$&0&0&0&$0$
\end{tabular}
\end{table}

\begin{corollary}\label{coro26090215}
The intersection of $\mathsf{V}(SR_6)$ and $\mathsf{V}(TR_6)$ is $\mathsf{V}(T_2)$.
\end{corollary}
\begin{proof}
First, by Proposition~\ref{subvariety-lattice} and Lemma~\ref{lem26090180},
$\mathsf{V}(T_2)$ is a subvariety of both $\mathsf{V}(SR_6)$ and $\mathsf{V}(TR_6)$,
and so $\mathsf{V}(T_2)$ is a subvariety of the intersection of $\mathsf{V}(SR_6)$ and $\mathsf{V}(TR_6)$.
Conversely, by Corollaries \ref{prop:define-SR6-ST7} and \ref{prop:define-TR6-ST7},
the intersection of $\mathsf{V}(SR_6)$ and $\mathsf{V}(TR_6)$ is
the subvariety of $\mathsf{V}(ST_7)$ defined by \eqref{id2608090150} and \eqref{eq3}.
From these two identities, one obtains \eqref{eq7}.
By Corollary~\ref{prop:define-T2-ST7}, it follows that
the intersection is a subvariety of $\mathsf{V}(T_2)$.
Thus the equality holds.
\end{proof}

In the remainder of this section, we introduce two important members of $\mathsf{V}(ST_7)$.
The first is $S_{(5, 2269)}$, whose Cayley tables are given in Table~\ref{tab:S52269}.
The following remark shows that $S_{(5,2269)}$ is a member of $\mathsf{V}(ST_7)$.

\begin{remark}\label{remark26083101}
$\mathsf{V}(S_{(5,2269)})=\mathsf{V}(S_c(ab),S_{(4,369)})$,
since $S_{(5,2269)}$ is isomorphic to a subdirect product of $S_c(ab)$ and $S_{(4,369)}$
via the congruences defined by the nontrivial blocks $\{1,5\}$ and $\{3,5\}$, respectively.
\end{remark}

\begin{table}[htbp]
\centering
\small
\caption{The Cayley tables of $S_{(5,2269)}$}\label{tab:S52269}
\renewcommand{\arraystretch}{1.08}
\setlength{\tabcolsep}{4.8pt}
\begin{tabular}{c|ccccc}
$+$&1&2&3&4&5\\ \hline
1&1&1&1&1&1\\
2&1&2&5&5&5\\
3&1&5&3&5&5\\
4&1&5&5&4&5\\
5&1&5&5&5&5
\end{tabular}
\qquad\qquad
\begin{tabular}{c|ccccc}
$\cdot$&1&2&3&4&5\\ \hline
1&1&1&1&1&1\\
2&1&1&1&3&1\\
3&1&1&1&1&1\\
4&1&3&1&1&1\\
5&1&1&1&1&1
\end{tabular}
\end{table}

For a term $\bt$, define the odd path closure $\mathbb{G}_{\bt}^{odd}$ by
\[
\mathbb{G}_{\bt}^{odd}
 =\{xy\mid \text{there is an odd path in $\mathbb{G}_{\bt}$ between $x$ and $y$}\}.
\]
The following result, due to Lyu et al.~\cite[Lemma~4.2]{LyuRenYue2026},
characterizes the inequalities of $S_c(ab)$.

\begin{lemma}\label{lemma26083101}
Let $\mathbf q \preceq \mathbf u$ be a nontrivial inequality, where
$\mathbf u=\mathbf u_1+\cdots+\mathbf u_n$ with $\mathbf u_i,\mathbf q\in X_c^+$ for $1\leq i \leq n$.
Then $\mathbf q\preceq\mathbf u$ is satisfied by $S_c(ab)$ if and only if one of the following conditions holds:
\begin{enumerate}[$(\rm i)$]
\item $\ell(\mathbf u_i)\geq3$ for some $\mathbf u_i\in\mathbf u$;
\item $c(L_1(\mathbf u))\cap c(L_2(\mathbf u))\neq\varnothing$;
\item the graph $\mathbb{G}_{\bu}$ contains an odd cycle;
\item $\mathbf q\in \mathbb{G}_\bu^{odd}$.
\end{enumerate}
\end{lemma}

From the inequality~\eqref{eq31}, one derives the inequalities
\begin{equation*}
 x_1x_{2n+2}\preceq x_1x_2+x_2x_3+\cdots+x_{2n+1}x_{2n+2}
 \qquad(n\geq1),
\end{equation*}
which we denote by $\delta_n$ $(n\geq1)$.
We now give a finite equational basis for $S_{(5, 2269)}$.
\begin{proposition}\label{prop:basis-S5}
The commutative ai-semiring variety $\mathsf{V}(S_{(5,2269)})$ is defined by
the identities \eqref{eq2}, \eqref{eq1}, \eqref{eq31}, and \eqref{eq:Cbridge}.
\end{proposition}
\begin{proof}
It is routine to check that $S_{(5,2269)}$ satisfies the identities~\eqref{eq2},
\eqref{eq1}, \eqref{eq31}, and \eqref{eq:Cbridge}.
It remains to show that every inequality satisfied by $S_{(5,2269)}$ is
derivable from \eqref{eq2},
\eqref{eq1}, \eqref{eq31}, and \eqref{eq:Cbridge}.
Let $\mathbf q\preceq\mathbf u$ be such a nontrivial inequality, where
$\mathbf u=\mathbf u_1+\cdots+\mathbf u_n$ with $\mathbf u_i, \bq\in X_c^+$ for $1 \leq i \leq n$.
By Remark~\ref{remark26083101}, $S_c(ab)$ satisfies $\mathbf q\preceq\mathbf u$.
By Lemma~\ref{lemma26083101}, one of the following cases occurs:
\begin{enumerate}[(\rm i)]
\item $\ell(\mathbf u_i)\geq3$ for some $\mathbf u_i\in\mathbf u$;
\item $c(L_1(\mathbf u))\cap c(L_2(\mathbf u))\neq\varnothing$;
\item the graph $\mathbb{G}_{\bu}$ contains an odd cycle;
\item $\bq\in \mathbb{G}_\bu^{odd}$.
\end{enumerate}

If Case~{\rm (i)} occurs, then
\[
\mathbf u\succeq\mathbf u_i
\overset{\eqref{eq1}}{\approx}x^2
\overset{\eqref{eq2}}{\succeq}\mathbf q.
\]

Suppose that Case~{\rm (iii)} occurs. Then $\mathbb{G}_{\bu}$ contains an odd cycle.
If this cycle is a loop, then $\bu$ contains $x^2$ for some $x \in X$, and so
\[
\mathbf u\succeq x^2 \overset{\eqref{eq2}}{\succeq}\mathbf q.
\]
If this cycle is of length $2k+1$ for some $k \geq 1$,
then $\bu$ has an additive subterm of the form
$x_1x_2+x_2x_3+\cdots+x_{2k}x_{2k+1}+x_{2k+1}x_{1}$, and so
\[
\mathbf u\succeq x_1x_2+x_2x_3+\cdots+x_{2k}x_{2k+1}+x_{2k+1}x_{1} \overset{\delta_k}{\succeq} x_1^2 \overset{\eqref{eq2}}{\succeq}\mathbf q.
\]

If Case~{\rm (iv)} occurs, then $\mathbf q\in \mathbb{G}_\bu^{odd}$.
We may write $\bq=xy$. Then there is a path of length $2k+1$ in $\mathbb{G}_\bu$ between $x$ and $y$ for some $k \geq 1$,
and so $\mathbf u\succeq\mathbf q$ is derivable from $\delta_k$.

In the remainder, we consider Case~{\rm (ii)}, and assume that
none of Cases~{\rm (i)}, {\rm (iii)}, or~{\rm (iv)} occurs.
Then $\bq \notin \mathbb{G}_\bu^{odd}$,
$\mathbf u=L_1(\mathbf u)+L_2(\mathbf u)$, $c(L_1(\mathbf u))\cap c(L_2(\mathbf u))\neq \emptyset$,
every word in $L_2(\mathbf u)$ is linear, and $\mathbb{G}_{\bu}$ is bipartite.
Choose a bipartition $(A, B)$ of $\mathbb{G}_{\bu}$ and define a semiring
homomorphism $\varphi \colon P_f(X_c^+) \to S_{(5, 2269)}$ by setting, for every
$t\in X$,
\[
\varphi(t)=
\begin{cases}
2,&t\in A,\\
4,&t\in B,\\
3,&t\in c(L_1(\mathbf u))\setminus c(L_2(\mathbf u)),\\
1,&\text{otherwise}.
\end{cases}
\]
Then $\varphi(L_2(\mathbf u))= 2\cdot 4=3$, and $\varphi(t)=2$ or $4$ for every $t \in c(L_1(\mathbf u))\cap c(L_2(\mathbf u))$.
Since $2+3=4+3=5$ in $S_{(5, 2269)}$,
it follows that $\varphi(L_1(\mathbf u))+\varphi(L_2(\mathbf u))=\varphi(L_1(\mathbf u))+3=5$,
and so
\[
\varphi(\mathbf u)=\varphi(L_1(\mathbf u)+L_2(\mathbf u))
=\varphi(L_1(\mathbf u))+\varphi(L_2(\mathbf u))=5,
\]
and so $\varphi(\mathbf q) \leq 5$.
Note that $1$ is the multiplicative zero and the additive maximum element of $S_{(5, 2269)}$.
Thus $\varphi(\mathbf q) \neq 1$, $c(\mathbf q)\subseteq c(\mathbf u)$, and $\mathbf q$ is a linear word of length at most two.

Suppose first that $\bq=xy$ for some distinct $x, y \in X$.
Then
\[
\varphi(\bq)=\varphi(xy)=\varphi(x)\varphi(y)\in \{1, 3\}.
\]
Since $\varphi(\mathbf q) \neq 1$, it follows that $\varphi(x)\varphi(y)=3$.
Since $3$ can only be the product of $2$ and $4$, we have $\{\varphi(x), \varphi(y)\}=\{2, 4\}$.
By the definition of $\varphi$, either $x\in A, y\in B$ or $x\in B, y\in A$.
Consequently, any path in $\mathbb{G}_{\mathbf u}$ connecting $x$ and $y$ must have odd length.

Without loss of generality, let $x\in A$ and $y\in B$. Define
\[
N(x)=\bigl\{z\in c(L_2(\mathbf u)) \mid z=x
\text{ or there is a path between $z$ and $x$}\bigr\},
\]
and define
$\psi\colon P_f(X_c^+)\to S_{(5,2269)}$ by
\[
\psi(t)=
\begin{cases}
2,&t\in(A\setminus N(x))\cup(B\cap N(x)),\\
4,&t\in(B\setminus N(x))\cup(A\cap N(x)),\\
3,&t\in c(L_1(\mathbf u))\setminus c(L_2(\mathbf u)),\\
1,&\text{otherwise}.
\end{cases}
\]
Then $\psi(x)=4$, $\psi(L_2(\bu))=3$, and $\psi(t)=2$ or $4$ for every $t \in c(L_1(\mathbf u))\cap c(L_2(\mathbf u))$.
Since $2+3=4+3=5$ in $S_{(5, 2269)}$,
it follows that
\[
\psi(L_1(\mathbf u))+\psi(L_2(\mathbf u))=\psi(L_1(\mathbf u))+3=5,
\]
and so
\[
\psi(\mathbf u)=\psi(L_1(\mathbf u)+L_2(\mathbf u))
=\psi(L_1(\mathbf u))+\psi(L_2(\mathbf u))=5.
\]
This implies that
\[
4\cdot \psi(y)=\psi(x)\cdot \psi(y)=\psi(xy)=\psi(\mathbf q)\leq \psi(\mathbf u)=5,
\]
and so $\psi(y)=2$. By the definition of $\psi$,
we obtain that $y\in B\cap N(x)$.
Thus $x$ and $y$ are joined by an odd path, and so $\bq\in \mathbb{G}_\bu^{odd}$,
contradicting the exclusion of Case (iv). Therefore,
$\ell(\mathbf q)=1$.

We can write $\mathbf q=z$ for some $z\in X$.
Then $z\in c(L_2(\mathbf u))$, and so $zt\in L_2(\mathbf u)$ for some $t\in X$.
Choose $x\in c(L_1(\mathbf u))\cap c(L_2(\mathbf u))$. Then $xy\in L_2(\mathbf u)$ for some $y\in X$.
Now we have
\[
\mathbf u\succeq x+xy+zt
\overset{\eqref{eq:Cbridge}}{\succeq}z=\mathbf q.
\]
This completes the proof.
\end{proof}

\begin{corollary}\label{prop:define-S5-ST7}
The variety $\mathsf{V}(S_{(5,2269)})$ is the subvariety of
$\mathsf{V}(ST_7)$ defined by the inequality~\eqref{eq31}.
\end{corollary}
\begin{proof}
This follows immediately from Propositions \ref{prop:basis-C7} and \ref{prop:basis-S5}.
\end{proof}

\begin{corollary}\label{coro26090138}
The intersection of $\mathsf{V}(S_{(5,2269)})$ and $\mathsf{V}(SR_6)$ is $\mathsf{V}(S_c(ab))$.
\end{corollary}
\begin{proof}
By \cite[Corollary 4.6]{LyuRenYue2026}, $\mathsf{V}(S_c(ab))$ is the subvariety of $\mathsf{V}(SR_6)$
defined by the identity \eqref{eq31}.
Note that \eqref{eq2}, \eqref{eq1}, and \eqref{eq:Cbridge} hold in $\mathsf{V}(SR_6)$.
By Proposition~\ref{prop:basis-S5}, the required result follows.
\end{proof}

\begin{corollary}\label{coro26090130}
The intersection of $\mathsf{V}(S_{(5,2269)})$ and $\mathsf{V}(TR_6)$ is $\mathsf{V}(S_{(4, 369)})$.
\end{corollary}
\begin{proof}
This is a direct consequence of Propositions \ref{Sstar-subvariety} and \ref{prop:basis-S5}.
\end{proof}

Let $SR_7$ denote the seven-element ai-semiring whose Cayley tables are given in Table~\ref{tab:N7}.
The following remark shows that $SR_7$ is also a member of $\mathsf{V}(ST_7)$.

\begin{remark}\label{remark26083102}
We note that $\mathsf{V}(SR_7)=\mathsf{V}(SR_6,S_{(4,369)})$,
since $SR_7$ is isomorphic to a subdirect
product of $SR_6$ and $S_{(4,369)}$ via two congruences: one with nontrivial block $\{1,7\}$,
and the other with nontrivial blocks $\{2,5\}$, $\{3,7\}$, and $\{4,6\}$.
\end{remark}

\begin{table}[htbp]
\centering
\caption{The Cayley tables of $SR_7$}\label{tab:N7}
\small
\renewcommand{\arraystretch}{1.08}
\setlength{\tabcolsep}{4.8pt}
\begin{tabular}{c|ccccccc}
$+$&1&2&3&4&5&6&7\\ \hline
1&1&1&1&1&1&1&1\\
2&1&2&7&7&2&7&7\\
3&1&7&3&7&7&7&7\\
4&1&7&7&4&7&4&7\\
5&1&2&7&7&5&7&7\\
6&1&7&7&4&7&6&7\\
7&1&7&7&7&7&7&7
\end{tabular}
\hspace{1.2cm}
\begin{tabular}{c|ccccccc}
$\cdot$&1&2&3&4&5&6&7\\ \hline
1&1&1&1&1&1&1&1\\
2&1&1&1&7&1&3&1\\
3&1&1&1&1&1&1&1\\
4&1&7&1&1&3&1&1\\
5&1&1&1&3&1&3&1\\
6&1&3&1&1&3&1&1\\
7&1&1&1&1&1&1&1
\end{tabular}
\end{table}

We shall use the following inequality:
\begin{align}
 x_1x_4&\preceq x_1+x_1x_2+x_2x_3+x_3x_4. \label{eq:Cpath}
\end{align}
By induction on $n$, the inequality~\eqref{eq:Cpath} yields
\begin{equation*}
 x_1x_{2n+2}\preceq  x_1+x_1x_2+x_2x_3+\cdots+x_{2n+1}x_{2n+2} \qquad (n \geq 1),
\end{equation*}
which we denote by $\tau_n$ $(n \geq 1)$.
The following result provides an infinite equational basis for $SR_7$.

\begin{proposition}\label{prop:basis-N7}
The commutative ai-semiring variety $\mathsf{V}(SR_7)$ is defined by the
identities \eqref{eq2}, \eqref{eq1}, \eqref{eq:Cbridge}, \eqref{eq:Cpath}, and $\sigma_n$ for all $n\geq1$.
\end{proposition}
\begin{proof}
The proof is somewhat lengthy and is divided into two parts.
We first show that $SR_7$ satisfies the
identities \eqref{eq2}, \eqref{eq1}, \eqref{eq:Cbridge}, \eqref{eq:Cpath}, and $\sigma_n$ for all $n\geq1$.
Indeed, it is easy to verify that $SR_7$ satisfies the identities~\eqref{eq2},
\eqref{eq1}, \eqref{eq:Cbridge}, and \eqref{eq:Cpath}.
By Remark~\ref{remark26082910}, $SR_6$ satisfies $\sigma_n$ for every
$n\geq1$. Since $S_{(4,369)}$ belongs to $\mathsf{V}(TR_6)$, it follows from
Lemma~\ref{R6identities} that $S_{(4,369)}$ satisfies $\sigma_n$ for every $n\geq1$.
By Remark~\ref{remark26083102}, we obtain that $SR_7$ also satisfies $\sigma_n$ for every $n\geq1$.

It remains to show that every inequality satisfied by $SR_7$ is
derivable from \eqref{eq2}, \eqref{eq1}, \eqref{eq:Cbridge}, \eqref{eq:Cpath}, and $\sigma_n$ for all $n\geq 1$.
Let $\mathbf q\preceq\mathbf u$ be such a nontrivial inequality, where
$\mathbf u=\mathbf u_1+\cdots+\mathbf u_n$ with $\mathbf u_i, \mathbf q\in X_c^+$ for all $1\leq i \leq n$.
Since $S_c(ab)$ lies in $\mathsf{V}(SR_6)$,
it follows from Remark~\ref{remark26083102} that $S_c(ab)$ belongs to $\mathsf{V}(SR_7)$,
and so $\mathbf q\preceq\mathbf u$ holds in $S_c(ab)$.
By Lemma~\ref{lemma26083101}, one of the following cases occurs:
\begin{enumerate}[(\rm i)]
\item $\ell(\mathbf u_i)\geq3$ for some $\mathbf u_i\in\mathbf u$;
\item $c(L_1(\mathbf u))\cap c(L_2(\mathbf u))\neq\emptyset$;
\item the graph $\mathbb{G}_{\bu}$ contains an odd cycle;
\item $\mathbf q\in \mathbb{G}_{\bu}^{odd}$.
\end{enumerate}

If Case~{\rm (i)} occurs, then
\[
\mathbf u\succeq\mathbf u_i
\overset{\eqref{eq1}}{\approx}x^2
\overset{\eqref{eq2}}{\succeq}\mathbf q.
\]

Suppose that Case~{\rm (iii)} occurs. Then $\mathbb{G}_{\bu}$ contains an odd cycle.
If this cycle is a loop, then $\bu$ contains $x^2$ for some $x \in X$, and so
\[
\mathbf u\succeq\mathbf u_i=x^2\overset{\eqref{eq2}}{\succeq}\mathbf q.
\]
If this cycle is of length $2k+1$ for some $k \geq 1$,
then $\bu$ has an additive subterm of the form
$x_1x_2+x_2x_3+\cdots+x_{2k}x_{2k+1}+x_{2k+1}x_{1}$, and so
\[
\mathbf u\succeq x_1x_2+x_2x_3+\cdots+x_{2k}x_{2k+1}+x_{2k+1}x_{1} \overset{\sigma_k}{\succeq} \bq.
\]

Now suppose that Case~{\rm (ii)} occurs but Case~{\rm (iv)} does not. Since
Case~{\rm (iii)} is excluded, we may also assume that $\mathbb{G}_{\bu}$ contains no odd cycle.
Then $\mathbb{G}_{\bu}$ is bipartite.
Choose a bipartition $(A, B)$ of $\mathbb{G}_{\bu}$ and define a semiring
homomorphism $\varphi:P_f(X_c^+)\to SR_7$ by
\[
\varphi(t)=
\begin{cases}
5,&t\in A,\\
6,&t\in B,\\
3,&t\in c(L_1(\mathbf u))\setminus c(L_2(\mathbf u)),\\
1,&\text{otherwise}.
\end{cases}
\]
Then $\varphi(L_2(\mathbf u))= 5\cdot 6=3$, and $\varphi(t)=5$ or $6$ for every $t \in c(L_1(\mathbf u))\cap c(L_2(\mathbf u))$.
Since $5+3=6+3=7$ in $SR_7$,
it follows that $\varphi(L_1(\mathbf u))+\varphi(L_2(\mathbf u))=\varphi(L_1(\mathbf u))+3=7$,
and so
\[
\varphi(\mathbf u)=\varphi(L_1(\mathbf u)+L_2(\mathbf u))
=\varphi(L_1(\mathbf u))+\varphi(L_2(\mathbf u))=7,
\]
and so $\varphi(\mathbf q) \leq 7$.
Note that $1$ is the multiplicative zero and the additive maximum element of $SR_7$.
Thus $\varphi(\mathbf q) \neq 1$, $c(\mathbf q)\subseteq c(\mathbf u)$, and $\mathbf q$ is a linear word of length at most two.

Suppose first that $\bq=xy$ for some distinct elements $x, y \in X$.
If $x\notin c(L_2(\bu))$, then $x\in c(L_1(\mathbf u))\setminus c(L_2(\mathbf u))$,
and so $\varphi(x)=3$. Furthermore, we have
\[
\varphi(\bq)=\varphi(xy)=\varphi(x)\varphi(y)=3\cdot \varphi(y)=1,
\]
a contradiction. Hence $x\in c(L_2(\bu))$. Similarly, $y\in c(L_2(\bu))$.
Thus $x, y \in A \cup B$.
If $x, y \in A$ or $x, y \in B$, then $\varphi(x)=\varphi(y)$,
and so
\[
\varphi(\bq)=\varphi(xy)=\varphi(x)\varphi(y)=\varphi(x)^2=1,
\]
a contradiction. Therefore, either $x\in A, y\in B$ or $x\in B, y\in A$.
Consequently, any path in $\mathbb{G}_{\mathbf u}$ connecting $x$ and $y$ must have odd length.

Without loss of generality, let $x\in A$ and $y\in B$. Define
\[
N(x)=\bigl\{z\in c(L_2(\mathbf u))\mid z=x
\text{ or there is a path between $z$ and $x$}\bigr\},
\]
and define
$\psi:P_f(X_c^+)\to SR_7$ by
\[
\psi(t)=
\begin{cases}
5,&t\in(A\setminus N(x))\cup(B\cap N(x)),\\
6,&t\in(B\setminus N(x))\cup(A\cap N(x)),\\
3,&t\in c(L_1(\mathbf u))\setminus c(L_2(\mathbf u)),\\
1,&\text{otherwise}.
\end{cases}
\]
Then $\psi(x)=6$ and $\psi(\mathbf u)=7$.
This implies that
\[
6\cdot \psi(y)=\psi(x)\cdot \psi(y)=\psi(xy)=\psi(\mathbf q)\leq \psi(\mathbf u)=7,
\]
and so $6\cdot \psi(y)=3$. By the multiplicative Cayley table of $SR_7$,
$\psi(y)=2$ or $5$. By the definition of $\psi$, $\psi(y)=5$, and so $y\in B\cap N(x)$.
Thus $x$ and $y$ are joined by an odd path, and so $\bq\in \mathbb{G}_\bu^{odd}$,
contradicting the exclusion of Case (iv). Therefore,
$\ell(\mathbf q)=1$.

We write $\bq=z$ for some $z\in X$.
Then $z\in c(L_2(\mathbf u))$, and so $zt\in L_2(\mathbf u)$ for some $t\in X$.
Choose $x\in c(L_1(\mathbf u))\cap c(L_2(\mathbf u))$. Then $xy\in L_2(\mathbf u)$ for some $y\in X$.
Now we have
\[
\mathbf u\succeq x+xy+zt
\overset{\eqref{eq:Cbridge}}{\succeq}z=\mathbf q.
\]

Finally, suppose that Case~{\rm (iv)} is true, and assume that neither Cases~{\rm (i)} nor {\rm (iii)} occurs.
Then $\bu=L_1(\bu)+L_2(\bu)$, $\mathbb{G}_{\bu}$ is bipartite, and $\mathbf q\in \mathbb{G}_{\bu}^{odd}$.
Hence $\ell(\bq)=2$,
and so $\bq$ is either a square of a variable or a linear word of length two.
If $\bq$ is a square of a variable, then $\bq=x^2$ for some $x \in X$,
and so $\mathbb{G}_{\mathbf{u}}$ contains an odd cycle through $x$, contradicting the exclusion of Case {\rm (iii)}.
Thus $\bq$ is a linear word of length two.
We write $\bq=xy$ for some distinct variables $x, y\in X$.
Since $\bq\in \mathbb{G}_{\bu}^{odd}$, there exists an odd path in $\mathbb{G}_{\bu}$ connecting $x$ and $y$.
Hence $L_2(\mathbf u)$ has an additive subterm of the form
\[
x_1x_2+x_2x_3+\cdots+x_{2m+1}x_{2m+2}
\]
for some $m\geq 1$, where $x_1=x$ and $x_{2m+2}=y$.

We claim that
$c(L_1(\mathbf u))\cap c(L_2(\mathbf u))\neq\emptyset$.
Suppose for contradiction that this is not true.
Since $\mathbb{G}_{\bu}$ is bipartite, one can choose a bipartition $(A, B)$ of $\mathbb{G}_{\bu}$
such that $x_1\in A$ and $x_{2m+2}\in B$.
Since the inequality $\bq\preceq \bu$ is nontrivial,
$x_1x_{2m+2}\notin L_2(\mathbf u)$. Define
$\theta: P_f(X_c^+)\to SR_7$ by
\[
\theta(v)=
\begin{cases}
2,&v=x_1,\\
4,&v=x_{2m+2},\\
5,&v\in A\setminus\{x_1\},\\
6,&v\in B\setminus\{x_{2m+2}\},\\
3,&v\in c(L_1(\mathbf u)),\\
1,&\text{otherwise}.
\end{cases}
\]
Then
\[
\theta(\bu)=\theta(L_1(\bu)+L_2(\bu))=\theta(L_1(\bu))+\theta(L_2(\bu))=3+3=3,
\]
whereas
\[
\theta(\mathbf q)=\theta(xy)=\theta(x_1x_{2m+2})=\theta(x_1)\theta(x_{2m+2})=2\cdot4=7.
\]
Since $7 \nleq 3$ in $SR_7$, it follows that $\theta(\mathbf q)\nleq \theta(\bu)$, a contradiction.
We have shown that $c(L_1(\mathbf u))\cap c(L_2(\mathbf u))\neq\emptyset$.

Now choose $z\in c(L_1(\mathbf u))\cap c(L_2(\mathbf u))$.
Then $zt\in L_2(\mathbf u)$ for some $t\in X$, and so
\[
\begin{aligned}
\mathbf u
&\succeq z+zt+x_1x_2+x_2x_3+\cdots+x_{2m+1}x_{2m+2}\\
&\overset{\eqref{eq:Cbridge}}{\succeq}
 x_1+x_1x_2+x_2x_3+\cdots+x_{2m+1}x_{2m+2}\\
&\overset{\tau_m}{\succeq}x_1x_{2m+2}=xy=\mathbf q.
\end{aligned}
\]
This completes the proof.
\end{proof}

\begin{corollary}\label{cor:N7-nfb}
The ai-semiring $SR_7$ is nonfinitely based.
\end{corollary}
\begin{proof}
Let $\Sigma$ denote the set $\{\eqref{eq2},\eqref{eq1},\eqref{eq:Cbridge},\eqref{eq:Cpath}\}\cup\{\sigma_n\mid n\geq1\}$.
By Proposition~\ref{prop:basis-N7}, $\Sigma$ is an equational basis of $SR_7$.
By Lemma~\ref{compactness}, it suffices to show that no finite subset of $\Sigma$ defines $\mathsf{V}(SR_7)$.

Let $\Sigma'$ be an arbitrary finite subset of $\Sigma$. Choose $n\geq2$
greater than every index $k$ for which $\sigma_k\in\Sigma'$, and put
\[
\Sigma_n=\{\eqref{eq2},\eqref{eq1},\eqref{eq:Cbridge},\eqref{eq:Cpath}\}
\cup\{\sigma_k\mid1\leq k<n\}.
\]
Then $\Sigma_n$ contains $\Sigma'$, but does not contain $\sigma_n$.
To show that $\Sigma'$ cannot derive $\sigma_n$,
we only need to prove $\Sigma_n$ cannot derive $\sigma_n$.

Indeed, it is easy to see that the term $\mathbf u^{(n)}$ is $x^2$-free,
$x_1x_2x_3$-free, and $(x+xy)$-free. Since $x+xy$ is a subterm of both
$x+xy+zt$ and $x_1+x_1x_2+x_2x_3+x_3x_4$, it follows from Lemma~\ref{freelemma} that
$\mathbf u^{(n)}$ is also free for these two terms. Moreover, by
\cite[Proposition~3.1]{LyuRenYue2026}, $\mathbf u^{(n)}$ is
$\mathbf u^{(k)}$-free for every $1\leq k<n$.
Consequently, for any identity in $\Sigma_n$ and any term $\bs$ occurring on either side of it,
there do not exist terms $\bp, \br$ and a substitution $\varphi$ such that
\[
\mathbf{u}^{(n)} = \bp \varphi(\bs) + \br.
\]
By Lemma~\ref{eqyield}, $\Sigma_n$ cannot derive
$\sigma_n$. Therefore, $SR_7$ is nonfinitely based.
\end{proof}

\begin{corollary}\label{prop:define-N7-ST7}
The variety $\mathsf{V}(SR_7)$ is the subvariety of $\mathsf{V}(ST_7)$
defined by the inequality~\eqref{eq:Cpath}.
\end{corollary}
\begin{proof}
This follows immediately from Propositions \ref{prop:basis-C7} and \ref{prop:basis-N7}.
\end{proof}

\section{The subvariety lattice of the variety $\mathsf{V}(ST_7)$}
Having established in the previous section that $\mathsf{V}(ST_7)$ is a nonfinitely based variety,
and having introduced its two important members $S_{(5,2269)}$ and $SR_7$,
it is natural to ask what its subvarieties are and how they are organized.
In this section, we describe the subvariety lattice of $\mathsf{V}(ST_7)$ in full detail.

\begin{proposition}\label{prop:minimal-ST7}
The variety $\mathsf{V}(T_2)$ is the only minimal nontrivial subvariety of
$\mathsf{V}(ST_7)$.
\end{proposition}
\begin{proof}
Since $T_2$ belongs to $\mathsf{V}(TR_6)$, it also belongs to
$\mathsf{V}(ST_7)$. By \cite[Theorem~1.1]{Polin1980} and
\cite[Table~1]{ShaoRen2015}, the minimal nontrivial subvarieties of the
variety of all ai-semirings are
\[
\mathsf{V}(L_2),\quad \mathsf{V}(R_2),\quad \mathsf{V}(M_2),\quad
\mathsf{V}(D_2),\quad \mathsf{V}(N_2),\quad \mathsf{V}(T_2).
\]
None of $L_2, R_2, M_2, D_2$, or $N_2$ satisfies the identity~\eqref{eq2}, which holds in $ST_7$.
Therefore, $\mathsf{V}(T_2)$ is the only minimal nontrivial subvariety of
$\mathsf{V}(ST_7)$.
\end{proof}

We now establish three exclusion criteria for subvarieties of $\mathsf{V}(ST_7)$.
These characterize, for certain algebras $A$,
precisely when a subvariety $\mathcal{V}$ of $\mathsf{V}(ST_7)$ does not contain $A$.

%\begin{proposition}\label{prop:exclude-T2-ST7}
%Let $\mathcal V$ be a subvariety of $\mathsf{V}(ST_7)$. Then $\mathcal V$
%does not contain $T_2$ if and only if $\mathcal V$ satisfies the identity
%\[
%x\approx x+x^2.
%\]
%\end{proposition}
%\begin{proof}
%The displayed identity fails in $T_2$, so one implication is immediate.
%Conversely, suppose that it fails in some $S\in\mathcal V$. Then
%$a\neq a+a^2$ for some $a\in S$. Identity~\eqref{eq2} gives
%$a+a^2=a^2$, and the subsemiring $\{a,a^2\}$ of $S$ is isomorphic to
%$T_2$. Hence $\mathcal V$ contains $T_2$.
%\end{proof}

\begin{proposition}\label{prop:exclude-S4369-ST7}
Let $\mathcal V$ be a subvariety of $\mathsf{V}(ST_7)$. Then $\mathcal V$
does not contain $S_{(4,369)}$ if and only if $\mathcal V$ satisfies the identity \eqref{id2608090150}.
\end{proposition}
\begin{proof}
Suppose that $\mathcal V$ satisfies \eqref{id2608090150}.
Then $\mathcal V$ does not contain $S_{(4,369)}$, since $S_{(4,369)}$ fails to satisfy \eqref{id2608090150}.
Conversely, assume that $\mathcal V$ does not satisfy \eqref{id2608090150}.
Then there exists an ai-semiring $S\in \mathcal{V}$ such that
$a^2 \nleq a+ab$
for some distinct elements $a, b \in S$.
Let $\langle a, b\rangle$ denote the subalgebra of $S$ generated by $a$ and $b$.
By \eqref{eq:Cbridge}, we have $a \leq b+ab$ and $b \leq a+ab$, so $a+ab=b+ab$.
It is routine to verify that
\[
\langle a, b\rangle = \{a, b, a^2, ab, a+b, a+ab\}.
\]
Consider the equivalence relation $\rho$ on $\langle a, b\rangle$ with equivalence classes
\[
R_1=\{a\},\quad R_2=\{b\},\quad R_3=\{a^2\},\quad
R_4=\{ab,a+b,a+ab\}.
\]
It is routine to check that $\rho$ is a semiring congruence on $\langle a,b\rangle$,
and that the quotient algebra $\langle a,b\rangle/\rho$ is isomorphic to $S_{(4,369)}$
under the mapping $R_1 \mapsto 3$, $R_2 \mapsto 4$, $R_3 \mapsto 2$, and $R_4 \mapsto 1$.
Thus $\mathcal V$ contains $S_{(4,369)}$.
\end{proof}

\begin{corollary}\label{coro26090226}
Let $\mathcal V$ be a subvariety of $\mathsf{V}(ST_7)$. Then $\mathcal V$
does not contain $S_{(4,369)}$ if and only if $\mathcal V$ is a subvariety of $\mathsf{V}(SR_6)$.
\end{corollary}
\begin{proof}
Assume that $\mathcal V$ does not contain $S_{(4,369)}$.
By Proposition~\ref{prop:exclude-S4369-ST7} and Corollary~\ref{prop:define-SR6-ST7},
$\mathcal{V}$ is a subvariety of $\mathsf{V}(SR_6)$.
Conversely, if $\mathcal{V}$ is a subvariety of $\mathsf{V}(SR_6)$,
then it cannot contain $S_{(4,369)}$, since \eqref{id2608090150} holds in $SR_6$ but fails in $S_{(4,369)}$.
\end{proof}

\begin{proposition}\label{prop:exclude-S48-ST7}
Let $\mathcal V$ be a subvariety of $\mathsf{V}(ST_7)$. Then $\mathcal V$
does not contain $S_c(ab)$ if and only if $\mathcal V$ satisfies
the inequality~\eqref{eq3}.
\end{proposition}
\begin{proof}
Suppose that $\mathcal V$ satisfies the inequality~\eqref{eq3}.
Then $\mathcal V$ does not contain $S_c(ab)$, since $S_c(ab)$ does not satisfy \eqref{eq3}.
Conversely, assume that $\mathcal V$ does not satisfy the inequality~\eqref{eq3}.
Then there exists an ai-semiring $S\in \mathcal{V}$ such that $a\nleq ab$ for some distinct elements $a,b\in S$.
Let $\langle a, b\rangle$ denote the subalgebra of $S$ generated by $a$ and $b$.
By \eqref{eq:Cbridge},
\[
\langle a, b\rangle=\{a, b, a^2, ab,a+b, a+ab\}.
\]
Consider the equivalence relation $\rho$ on $\langle a,b\rangle$ with equivalence classes
\[
R_1=\{a\}, \quad R_2=\{b\}, \quad R_3=\{ab\},\quad R_4=\{a^2,a+b,a+ab\}.
\]
Now it is easy to verify that $\rho$ is a semiring
congruence on $\langle a,b\rangle$, and that
the quotient algebra $\langle a,b\rangle/\rho$ is isomorphic to $S_c(ab)$
under the mapping $R_1\mapsto a$, $R_2\mapsto b$, $R_3\mapsto ab$, and $R_4\mapsto 0$.
Hence $\mathcal V$ contains $S_c(ab)$.
\end{proof}

\begin{corollary}\label{coro26090225}
Let $\mathcal V$ be a subvariety of $\mathsf{V}(ST_7)$. Then $\mathcal V$
does not contain $S_c(ab)$ if and only if $\mathcal V$ is a subvariety of $\mathsf{V}(TR_6)$.
\end{corollary}
\begin{proof}
Suppose that $\mathcal V$ does not contain $S_c(ab)$.
By Proposition~\ref{prop:exclude-S48-ST7} and Corollary~\ref{prop:define-TR6-ST7},
$\mathcal{V}$ is a subvariety of $\mathsf{V}(TR_6)$.
Conversely, if $\mathcal{V}$ is a subvariety of $\mathsf{V}(TR_6)$,
then it cannot contain $S_c(ab)$, since \eqref{eq3} holds in $TR_6$ but fails in $S_c(ab)$.
\end{proof}

\begin{proposition}\label{prop:exclude-TR6-ST7}
Let $\mathcal V$ be a subvariety of $\mathsf{V}(ST_7)$. Then $\mathcal V$
does not contain $TR_6$ if and only if $\mathcal V$ satisfies
the inequality~\eqref{eq:Cpath}.
\end{proposition}
\begin{proof}
The proof is substantially more involved than those of the preceding two propositions.
Assume that $\mathcal V$ satisfies the inequality~\eqref{eq:Cpath}.
Then $\mathcal V$ does not contain $TR_6$, since $TR_6$ fails to satisfy \eqref{eq:Cpath}.

For the converse, suppose that $\mathcal{V}$ does not contain $TR_6$.
If $\mathcal{V}$ is the trivial variety, then $\mathcal{V}$ does not contain $TR_6$.
Now assume that $\mathcal{V}$ is nontrivial. Then
there is a nontrivial inequality
$\bq\preceq\bu$ that holds in $\mathcal{V}$ but fails in $TR_6$, where
$\bu=\bu_1+\cdots+\bu_n$ with $\bu_i,\bq\in X_c^+$ for all $1\leq i\leq n$.
One can obtain that $L_2(\bu)$ is nonempty.
By Proposition~\ref{wordprop}, none of the following four conditions holds:
\begin{enumerate}[$(\rm i)$]
\item $\ell(\mathbf{u}_i)\geq3$ for some $\mathbf{u}_i\in\mathbf{u}$;
\item $\mathbf{u}$ contains a non-linear word;
\item the graph $\mathbb{G}_{\mathbf{u}}$ contains an odd cycle;
\item $\ell(\mathbf{q})=1$ and $c(\mathbf{q})\subseteq c(\mathbf{u})$.
\end{enumerate}
Hence
$\bu=L_1(\bu)+L_2(\bu)$, every word in $L_2(\bu)$ is linear,
$\mathbb{G}_{\bu}$ is bipartite, and either $\ell(\mathbf{q})\neq 1$ or $c(\mathbf{q})\nsubseteq c(\mathbf{u})$.
Let $(A,B)$ be a bipartition of $\mathbb{G}_{\bu}$. Then $A\cup B=c(L_2(\bu))$.
Consider the following four cases.

{\bf Case 1.} $\bq=z$ for some $z\in X$. Then $z\notin c(\bu)$.
Let us define a substitution $\varphi \colon P_f(X_c^+)\to P_f(X_c^+)$ by
\[
\varphi(t)=
\begin{cases}
x_1x_4,&t=z,\\
x_2,&t\in A,\\
x_3,&t\in B,\\
x_1,&\text{otherwise}.
\end{cases}
\]
Then $\varphi(\bq)=x_1x_4$, $\varphi(L_2(\bu))=x_2x_3$, and $\varphi(L_1(\bu))$ is an additive subterm of $x_1+x_2+x_3$.
Note that
\[
\varphi(\bu)=\varphi(L_1(\bu)+L_2(\bu))=\varphi(L_1(\bu))+\varphi(L_2(\bu)).
\]
It follows that $\varphi(\bu)$ is an additive subterm of $x_1+x_2+x_3+x_1x_2+x_2x_3+x_3x_4$.
Now we have
\[
\begin{aligned}
x_1x_4=\varphi(\bq)
&\preceq \varphi(\bu) \\
&\preceq x_1+x_2+x_3+x_1x_2+x_2x_3+x_3x_4 \\
&\overset{\eqref{eq:Cbridge}}{\approx} x_1+x_1x_2+x_2x_3+x_3x_4.
\end{aligned}
\]
This derives the inequality \eqref{eq:Cpath}.
Since $\varphi(\bq) \preceq \varphi(\bu)$ and \eqref{eq:Cbridge} hold in $\mathcal V$,
we obtain that $\mathcal{V}$ satisfies \eqref{eq:Cpath}.

{\bf Case 2.} $\bq=xy$ for some distinct variables $x,y\in X$, where
either $x\in A$, $y\in B$ or $x\in B$, $y\in A$.
Without loss of generality, assume that $x\in A$ and $y\in B$.
Since the inequality $\bq\preceq\bu$ is nontrivial, we have
$xy\notin L_2(\bu)$. Define a substitution
$\varphi:P_f(X_c^+)\to P_f(X_c^+)$ by
\[
\varphi(t)=
\begin{cases}
	x_1,&t=x,\\
	x_4,&t=y,\\
	x_3,&t\in A\setminus\{x\},\\
	x_2,&t\in B\setminus\{y\},\\
	x_1,&\text{otherwise}.
\end{cases}
\]
Then $\varphi(\bq)=x_1x_4$, $\varphi(L_1(\bu))$ is an additive subterm of
$x_1+x_2+x_3+x_4$,
and $\varphi(L_2(\bu))$ is an additive subterm of $x_1x_2+x_2x_3+x_3x_4$.
Note that
\[
\varphi(\bu)
=\varphi(L_1(\bu)+L_2(\bu))
=\varphi(L_1(\bu))+\varphi(L_2(\bu)).
\]
It follows that $\varphi(\bu)$ is an additive subterm of
\[
x_1+x_2+x_3+x_4+x_1x_2+x_2x_3+x_3x_4.
\]
The remainder of the argument is identical to Case 1 and is therefore omitted.

{\bf Case 3.} $\bq$ satisfies one of the following two conditions:
\begin{itemize}
\item $\bq=xy$ for some distinct variables $x,y\in X$,
where $x,y\in A$, or $x,y\in B$, or $x,y \notin A\cup B$,
or one of $x,y$ does not belong to $A\cup B$, the other belongs to $A$.

\item $\bq$ is either a square of a variable or a word of length at least three.
\end{itemize}
Define a substitution $\varphi\colon P_f(X_c^+)\to P_f(X_c^+)$ by
\[
\varphi(t)=
\begin{cases}
	x_3,&t\in B,\\
	x_2,&\text{otherwise}.
\end{cases}
\]
Then $\varphi(L_1(\bu))$ is an additive subterm of $x_2+x_3$ and $\varphi(L_2(\bu))=x_2x_3$.
Since
\[
\varphi(\bu)
=\varphi(L_1(\bu)+L_2(\bu))
=\varphi(L_1(\bu))+\varphi(L_2(\bu)),
\]
it follows that $\varphi(\bu)$ is an additive subterm of
\[
x_1+x_2+x_3+x_1x_2+x_2x_3+x_3x_4.
\]
Moreover, we have that $\varphi(\bq)$ is either a square of a variable or a word of length at least three.
Now we have
\[
\begin{aligned}
x_1x_4
&\overset{\eqref{eq2}}{\preceq} x^2
\overset{\eqref{eq1}}{\approx} \varphi(\bq)
\preceq \varphi(\bu) \\
&\preceq x_1+x_2+x_3+x_1x_2+x_2x_3+x_3x_4 \\
&\overset{\eqref{eq:Cbridge}}{\approx} x_1+x_1x_2+x_2x_3+x_3x_4.
\end{aligned}
\]
Since \eqref{eq1}, \eqref{eq:Cbridge}, and $\varphi(\bq) \preceq \varphi(\bu)$ hold in $\mathcal V$,
we obtain that $\mathcal{V}$ satisfies \eqref{eq:Cpath}.

{\bf Case 4.}
$\bq=xy$ for some distinct variables $x,y\in X$,
where one of $x,y$ does not belong to $A\cup B$, the other belongs to $B$.
Without loss of generality, we assume that $x\notin A\cup B$ and $y \in B$.
Define a substitution $\varphi\colon P_f(X_c^+)\to P_f(X_c^+)$ by
\[
\varphi(t)=
\begin{cases}
	x_1,&t=x,\\
	x_4,&t=y,\\
	x_3,&t\in A,\\
	x_2,&t\in B\setminus\{y\},\\
	x_1,&\text{otherwise}.
\end{cases}
\]
Then $\varphi(\bq)=x_1x_4$,
$\varphi(L_1(\bu))$ is an additive subterm of
$x_1+x_2+x_3+x_4$,
and $\varphi(L_2(\bu))$ is an additive subterm of $x_1x_2+x_2x_3+x_3x_4$.
Note that
\[
\varphi(\bu)
=\varphi(L_1(\bu)+L_2(\bu))
=\varphi(L_1(\bu))+\varphi(L_2(\bu)).
\]
It follows that $\varphi(\bu)$ is an additive subterm of
\[
x_1+x_2+x_3+x_4+x_1x_2+x_2x_3+x_3x_4.
\]
The remainder of the argument is identical to Case 1 and is therefore omitted.

This completes the proof.
\end{proof}

\begin{corollary}\label{coro26090270}
Let $\mathcal V$ be a subvariety of $\mathsf{V}(ST_7)$. Then $\mathcal V$
does not contain $TR_6$ if and only if $\mathcal V$ is a subvariety of $\mathsf{V}(SR_7)$.
\end{corollary}
\begin{proof}
Suppose that $\mathcal V$ does not contain $TR_6$.
By Proposition~\ref{prop:exclude-TR6-ST7} and Corollary~\ref{prop:define-N7-ST7},
$\mathcal{V}$ is a subvariety of $\mathsf{V}(SR_7)$.
Conversely, if $\mathcal{V}$ is a subvariety of $\mathsf{V}(SR_7)$,
then it cannot contain $TR_6$, since \eqref{eq:Cpath} holds in $SR_7$ but fails in $TR_6$.
\end{proof}

In the following we determine the subvariety lattice $\mathcal{L}(\mathsf{V}(S_{(5,2269)}))$ of the variety $\mathsf{V}(S_{(5,2269)})$.
\begin{proposition}\label{pro06090152259}
The lattice $\mathcal{L}(\mathsf{V}(S_{(5,2269)}))$ consists of five varieties:
$\mathsf{V}(S_{(5,2269)})$, $\mathsf{V}(S_c(ab))$, $\mathsf{V}(S_{(4, 369)})$,
$\mathsf{V}(T_2)$, and $\mathbf{T}$.
\end{proposition}
\begin{proof}
By Remark~\ref{remark26083101}, $\mathsf{V}(S_{(5,2269)})=\mathsf{V}(S_c(ab), S_{(4,369)})$.
Now let $\mathcal{V}$ be an arbitrary proper subvariety of $\mathsf{V}(S_{(5,2269)})$.
Then $\mathcal{V}$ cannot simultaneously contain $S_c(ab)$ and $S_{(4,369)}$.

If $\mathcal{V}$ does not contain $S_c(ab)$, then by Corollary~\ref{coro26090225},
$\mathcal{V}$ is a subvariety of $\mathsf{V}(TR_6)$.
Hence $\mathcal{V}$ is a subvariety of the variety $\mathsf{V}(S_{(5,2269)})\cap\mathsf{V}(TR_6)$.
Now it follows from Corollary~\ref{coro26090130} that $\mathcal{V}$ is a subvariety of $\mathsf{V}(S_{(4, 369)})$.
By Proposition~\ref{subvariety-lattice}, $\mathcal{V}$ coincides with $\mathsf{V}(S_{(4, 369)})$,
$\mathsf{V}(T_2)$ or $\mathbf{T}$.

If $\mathcal{V}$ does not contain $S_{(4,369)}$, then by Corollary~\ref{coro26090226},
$\mathcal{V}$ is a subvariety of $\mathsf{V}(SR_6)$.
Thus $\mathcal{V}$ is a subvariety of the variety $\mathsf{V}(S_{(5,2269)})\cap\mathsf{V}(SR_6)$.
By Corollary~\ref{coro26090138}, $\mathcal{V}$ is a subvariety of $\mathsf{V}(S_c(ab))$.
By Lemma~\ref{lem26090180}, $\mathcal{V}$ coincides with $\mathsf{V}(S_c(ab))$, $\mathsf{V}(T_2)$ or $\mathbf{T}$.
This completes the proof.
\end{proof}

To determine the lattice $\mathcal{L}(\mathsf{V}(ST_7))$ of subvarieties of $\mathsf{V}(ST_7)$,
it is natural to consider the following restriction mapping
\begin{equation}\label{mp26090188}
\pi \colon \mathcal{L}(\mathsf{V}(ST_7))\to \mathcal{L}(\mathsf{V}(S_{(5,2269)})), \quad \mathcal{V} \mapsto  \mathcal{V}\cap\mathsf{V}(S_{(5,2269)}).
\end{equation}
This mapping is surjective.
By Proposition~\ref{pro06090152259},
the preimages of $\pi$ give the following partition of $\mathcal{L}(\mathsf{V}(ST_7))$:
$\pi^{-1}(\mathbf{T})$, $\pi^{-1}(\mathsf{V}(T_2))$, $\pi^{-1}(\mathsf{V}(S_{(4, 369)}))$,
$\pi^{-1}(\mathsf{V}(S_c(ab)))$, and $\pi^{-1}(\mathsf{V}(S_{(5,2269)}))$.

The next proposition describes each preimage explicitly.
\begin{proposition}\label{prop:ST7-fibres}
Let $\pi$ be the mapping defined in \eqref{mp26090188}. Then
\begin{enumerate}[$(1)$]
\item $\pi^{-1}(\mathbf T)=\{\mathbf T\}$.
\item $\pi^{-1}(\mathsf{V}(T_2))=\{\mathsf{V}(T_2)\}$.
\item
$\pi^{-1}(\mathsf{V}(S_{(4, 369)}))
 =\{\mathsf{V}(S_{(4, 369)}),\mathsf{V}(TR_6)\}$.
\item
$\pi^{-1}(\mathsf{V}(S_c(ab)))
 =\{\mathsf{V}(S_c(ab)),\mathsf{V}(SR_6)\}$.
\item
$\pi^{-1}(\mathsf{V}(S_{(5,2269)}))
 =\{\mathsf{V}(S_{(5,2269)}),\mathsf{V}(SR_7),
 \mathsf{V}(ST_7)\}$.
\end{enumerate}
\end{proposition}
\begin{proof}
We prove the five assertions in order. The first assertion is straightforward.
Assertions (2), (3), and (4) rely on some results established earlier.
The final assertion is the most involved and will be treated in detail.

$(1)$ This is a direct consequence of Proposition~\ref{prop:minimal-ST7}.

$(2)$ It is easy to see that $\mathsf{V}(T_2)$ lies in $\pi^{-1}(\mathsf{V}(T_2))$.
For the converse, let $\mathcal{V}$ be an arbitrary variety in $\pi^{-1}(\mathsf{V}(T_2))$.
Then $\mathcal{V} \cap \mathsf{V}(S_{(5,2269)}) = \mathsf{V}(T_2)$;
in particular, $\mathcal{V}$ contains $\mathsf{V}(T_2)$.
Moreover, by Proposition~\ref{pro06090152259}, $\mathcal{V}$ contains neither $S_{(4,369)}$ nor $S_c(ab)$.
Corollaries~\ref{coro26090226} and~\ref{coro26090225} then imply that
$\mathcal{V}$ is a subvariety of $\mathsf{V}(SR_6) \cap \mathsf{V}(TR_6)$,
which equals $\mathsf{V}(T_2)$ by Corollary~\ref{coro26090215}.
Consequently, $V=\mathsf{V}(T_2)$.
Therefore, $\pi^{-1}(\mathsf{V}(T_2)) = \{\mathsf{V}(T_2)\}$.

$(3)$ First, $\{\mathsf{V}(S_{(4, 369)}),\mathsf{V}(TR_6)\}$ is contained in $\pi^{-1}(\mathsf{V}(S_{(4, 369)}))$,
since $\mathsf{V}(TR_6)$ is a subvariety of $\mathsf{V}(ST_7)$ that contains $\mathsf{V}(S_{(4, 369)})$.
Conversely, if $\mathcal{V}$ is a variety in $\pi^{-1}(\mathsf{V}(S_{(4, 369)}))$,
then $\mathcal{V} \cap \mathsf{V}(S_{(5,2269)}) = \mathsf{V}(S_{(4, 369)})$;
in particular, $\mathcal{V}$ contains $\mathsf{V}(S_{(4, 369)})$.
Moreover, by Proposition~\ref{pro06090152259}, $\mathcal{V}$ does not contain $S_c(ab)$.
Corollary~\ref{coro26090225} then implies that $\mathcal{V}$ is a subvariety of $\mathsf{V}(TR_6)$.
Thus $\mathcal{V}$ is a subvariety of $\mathsf{V}(TR_6)$ that contains $\mathsf{V}(S_{(4, 369)})$.
By Proposition~\ref{subvariety-lattice}, $\mathcal{V}$ is either $\mathsf{V}(TR_6)$ or $\mathsf{V}(S_{(4, 369)})$.
Consequently, $\pi^{-1}(\mathsf{V}(S_{(4, 369)}))=\{\mathsf{V}(S_{(4, 369)}),\mathsf{V}(TR_6)\}$.

$(4)$ Both $\mathsf{V}(S_c(ab))$ and $\mathsf{V}(SR_6)$ lie in $\pi^{-1}(\mathsf{V}(S_c(ab)))$,
since $\mathsf{V}(SR_6)$ is a subvariety of $\mathsf{V}(ST_7)$ containing $\mathsf{V}(S_c(ab))$.
For the converse, let $\mathcal{V}$ be an arbitrary variety in $\pi^{-1}(\mathsf{V}(S_c(ab)))$.
Then $\mathcal{V} \cap \mathsf{V}(S_{(5,2269)}) = \mathsf{V}(S_c(ab))$.
In particular, $\mathcal{V}$ contains $\mathsf{V}(S_c(ab))$.
Moreover, by Proposition~\ref{pro06090152259}, $\mathcal{V}$ does not contain $S_{(4, 369)}$.
Corollary~\ref{coro26090226} then implies that $\mathcal{V}$ is a subvariety of $\mathsf{V}(SR_6)$.
Hence $\mathcal{V}$ is a subvariety of $\mathsf{V}(SR_6)$ that contains $\mathsf{V}(S_c(ab))$.
By Lemma~\ref{lem26090180}, $\mathcal{V}$ is either $\mathsf{V}(SR_6)$ or $\mathsf{V}(S_c(ab))$.
Therefore, $\pi^{-1}(\mathsf{V}(S_c(ab)))=\{\mathsf{V}(S_c(ab)),\mathsf{V}(SR_6)\}$.

$(5)$ First, both $\mathsf{V}(S_{(5,2269)})$ and $\mathsf{V}(ST_7)$
lie in $\pi^{-1}(\mathsf{V}(S_{(5,2269)}))$.
By Remarks~\ref{remark26083101} and~\ref{remark26083102},
together with the fact that $S_c(ab) \in \mathsf{V}(SR_6)$,
we have that $\mathsf{V}(SR_7)$ contains $\mathsf{V}(S_{(5, 2269)})$,
so $\mathsf{V}(SR_7)$ also belongs to $\pi^{-1}(\mathsf{V}(S_{(5,2269)}))$.

For the converse, let $\mathcal{V}$ be an arbitrary variety in $\pi^{-1}(\mathsf{V}(S_{(5,2269)}))$.
Then the intersection $\mathcal{V} \cap \mathsf{V}(S_{(5,2269)})=\mathsf{V}(S_{(5,2269)})$,
and so $\mathcal{V}$ contains $\mathsf{V}(S_{(5,2269)})$.
Hence $\mathcal{V}$ is a subvariety of $\mathsf{V}(ST_7)$ that contains $\mathsf{V}(S_{(5,2269)})$.
If $\mathcal{V}$ is a proper subvariety of $\mathsf{V}(ST_7)$,
then there is a nontrivial inequality $\bq \preceq \bu$ that holds in $\mathcal{V}$ but fails in $\mathsf{V}(ST_7)$.

We claim that $\mathcal{V}$ does not contain $TR_6$.
Suppose, to the contrary, that it does.
Then $TR_6$ satisfies $\bq \preceq \bu$.
By Proposition~\ref{wordprop}, $\bu$ and $\bq$ satisfy one of the following four conditions:
\begin{enumerate}[$(\rm i)$]
\item $\ell(\mathbf{u}_i)\geq3$ for some $\mathbf{u}_i\in\mathbf{u}$;
\item $\mathbf{u}$ contains a non-linear word;
\item the graph $\mathbb{G}_{\mathbf{u}}$ contains an odd cycle;
\item $\ell(\mathbf{q})=1$ and $c(\mathbf{q})\subseteq c(\mathbf{u})$.
\end{enumerate}
If {(\rm i)} or {(\rm ii)} holds, then $\bq \preceq \bu$ is derivable from \eqref{eq2} and \eqref{eq1}, which are satisfied by $ST_7$.
Hence $\bq \preceq \bu$ holds in $ST_7$, contradicting the choice of $\bq \preceq \bu$.
If {(\rm iii)} holds, then $\bq \preceq \bu$ is derivable from some $\sigma_n$, which holds in $ST_7$.
Thus $\bq \preceq \bu$ is satisfied by $ST_7$, again a contradiction.
Thus (iv) must hold, and none of (i), (ii), or (iii) is true.
Therefore, $\ell(\mathbf{q})=1$, $c(\mathbf{q})\subseteq c(\mathbf{u})$,
$\bu=L_1(\bu)+L_2(\bu)$, every word in $L_2(\mathbf u)$ is linear,
and $\mathbb G_{\mathbf u}$ is bipartite.
We write $\bq=z$. Then $z\in c(L_2(\bu))$, and so $zt\in L_2(\bu)$ for some $t \in X$.

We claim that $c(L_1(\mathbf u)) \cap c(L_2(\mathbf u)) = \emptyset$.
Suppose otherwise, and choose $x \in c(L_1(\mathbf u)) \cap c(L_2(\mathbf u))$. Then
\[
\mathbf u \succeq x + xy + zt
 \overset{\eqref{eq:Cbridge}}{\succeq} z = \mathbf q.
\]
Since \eqref{eq:Cbridge} holds in $ST_7$, this would imply that $\bq \preceq \bu$ is satisfied by $ST_7$, contradicting the choice of $\bq \preceq \bu$.
Thus the intersection is empty.

Let $(A, B)$ be a bipartition of $\mathbb G_{\mathbf u}$. Then $z\in A$ or $z\in B$.
Without loss of generality, assume that $z\in A$.
Define a substitution $\varphi: P_f(X_c^+)\to P_f(X_c^+)$ by
\[
\varphi(t)=
\begin{cases}
x,&t\in A,\\
y,&t\in B,\\
xy,&\text{otherwise}.
\end{cases}
\]
(Note that this is well-defined since $c(L_1(\mathbf u)) \cap c(L_2(\mathbf u)) = \emptyset$.)
Then
\[
\varphi(\mathbf u)=\varphi(L_1(\bu))+\varphi(L_2(\bu))=xy+xy=xy
\]
and
$\varphi(\mathbf q)=x$.
Hence $\varphi(\mathbf q)\preceq \varphi(\mathbf u)$ is precisely the identity \eqref{eq3}.
Since $\mathcal V$ satisfies $\mathbf q\preceq\mathbf u$,
it also satisfies $\varphi(\mathbf q) \preceq \varphi(\mathbf u)$, hence \eqref{eq3}.
Corollary~\ref{prop:define-TR6-ST7} then implies that $\mathcal V$ is a subvariety of $\mathsf{V}(TR_6)$.
On the other hand, since $\mathsf{V}(TR_6)$ is a subvariety of $\mathcal V$, it now follows that $\mathcal V = \mathsf{V}(TR_6)$.
But then, by Corollary~\ref{coro26090130},
\[
\pi(\mathcal V)=\mathcal{V}\cap\mathsf{V}(S_{(5,2269)})
 =\mathsf{V}(TR_6)\cap\mathsf{V}(S_{(5,2269)})
 =\mathsf{V}(S_{(4,369)}),
\]
which is not $\mathsf{V}(S_{(5,2269)})$, contradicting $\mathcal V \in \pi^{-1}(\mathsf{V}(S_{(5,2269)}))$.
Therefore, $\mathcal{V}$ does not contain $TR_6$.

By Corollary~\ref{coro26090270}, $\mathcal{V}$ is a subvariety of $\mathsf{V}(SR_7)$.
If $\mathcal{V}$ is a proper subvariety of $\mathsf{V}(SR_7)$,
then there is a
nontrivial inequality $\mathbf q\preceq\mathbf u$ that holds
in $\mathcal V$ but fails in $SR_7$.
Since $\mathcal{V}$ contains $\mathsf{V}(S_{(5,2269)})$, and $\mathsf{V}(S_{(5,2269)})$ contains $S_c(ab)$,
it follows that $\mathbf q\preceq\mathbf u$ is satisfied by both $S_{(5,2269)}$ and $S_c(ab)$.
By Lemma~\ref{lemma26083101}, one of the following conditions holds:
\begin{enumerate}[$(\rm i)$]
\item $\ell(\mathbf u_i)\geq3$ for some $\mathbf u_i\in\mathbf u$;
\item $c(L_1(\mathbf u))\cap c(L_2(\mathbf u))\neq\emptyset$;
\item the graph $\mathbb{G}_{\bu}$ contains an odd cycle;
\item $\mathbf q\in \mathbb{G}_\bu^{odd}$.
\end{enumerate}

If (i) holds, then $\bq \preceq \bu$ is derivable from \eqref{eq2} and \eqref{eq1}, which are satisfied by $SR_7$.
Hence $\bq \preceq \bu$ holds in $SR_7$, contradicting the choice of $\bq \preceq \bu$.
If {(\rm iii)} holds, then $\bq \preceq \bu$ is derivable from some $\sigma_n$, which holds in $SR_7$.
Thus $\bq \preceq \bu$ is satisfied by $SR_7$, again a contradiction.

If (ii) holds while (i), (iii), and (iv) all fail, then, as in the proof of Proposition~\ref{prop:basis-S5},
$\bq \preceq \bu$ is derivable from \eqref{eq:Cbridge}, which holds in $SR_7$.
Hence $\bq \preceq \bu$ is satisfied by $SR_7$.
This again contradicts the choice of $\bq \preceq \bu$.

Now suppose that (ii) and (iv) simultaneously hold, and neither (i) nor (iii) is true.
Then $c(L_1(\mathbf u))\cap c(L_2(\mathbf u))\neq\emptyset$, $\mathbf q\in \mathbb{G}_\bu^{odd}$,
$\bu=L_1(\bu)+L_2(\bu)$, and $\mathbb{G}_{\bu}$ is bipartite; in particular, $\ell(\bq)=2$.
We claim that $\bq$ is linear. If this is not true, then $\bq=x^2$ for some $x\in X$.
Let $(A, B)$ be a bipartition of $\mathbb G_{\mathbf u}$. Then $x\in A$ or $x\in B$.
Without loss of generality, assume that $x\in A$.
Define a semiring homomorphism $\varphi: P_f(X_c^+)\to S_{(5, 2269)}$ by
\[
\varphi(t)=
\begin{cases}
2,&t\in A,\\
4,&t\in B,\\
5,&\text{otherwise}.
\end{cases}
\]
Then $\varphi(\bu)=5$ and $\varphi(\bq)=1$. Since $1\nleq 5$ in $S_{(5, 2269)}$,
it follows that $\varphi(\bq)\nleq \varphi(\bu)$,
which contradicts the fact that $\bq \preceq \bu$ holds in $S_{(5, 2269)}$.
Thus $\bq$ is a linear word of length two.

We write $\bq=xy$ for distinct variables $x, y\in X$.
Since $\bq\in \mathbb{G}_{\bu}^{odd}$, there is an odd path in $\mathbb{G}_{\bu}$ connecting $x$ and $y$.
Hence $L_2(\mathbf u)$ has an additive subterm of the form
\[
x_1x_2+x_2x_3+\cdots+x_{2m+1}x_{2m+2}
\]
for some $m\geq 1$, where $x_1=x$ and $x_{2m+2}=y$.
Since $c(L_1(\mathbf u))\cap c(L_2(\mathbf u))\neq\emptyset$,
we may choose $z\in c(L_1(\mathbf u))\cap c(L_2(\mathbf u))$.
Then $zt\in L_2(\mathbf u)$ for some $t\in X$, and so
\[
\begin{aligned}
\mathbf u
&\succeq z+zt+x_1x_2+x_2x_3+\cdots+x_{2m+1}x_{2m+2}\\
&\overset{\eqref{eq:Cbridge}}{\succeq}
 x_1+x_1x_2+x_2x_3+\cdots+x_{2m+1}x_{2m+2}\\
&\overset{\tau_m}{\succeq}\mathbf q.
\end{aligned}
\]
This shows that $\bq \preceq \bu$ is derivable from \eqref{eq:Cbridge} and $\tau_m$,
both of which hold in $SR_7$. Hence $\bq \preceq \bu$ is satisfied by $SR_7$,
contradicting the choice of $\bq \preceq \bu$.

Finally, (iv) holds, and none of (i), (ii), or (iii) is true.
Then $\mathbf q\in \mathbb{G}_\bu^{odd}$, $\bu=L_1(\bu)+L_2(\bu)$, $c(L_1(\mathbf u))\cap c(L_2(\mathbf u))=\emptyset$,
and $\mathbb{G}_{\bu}$ is bipartite.
Hence $\bq$ is linear.
We write $\bq=xy$ for distinct variables $x, y\in X$.
Since $\bq\in \mathbb{G}_{\bu}^{odd}$, there exists an odd path in $\mathbb{G}_{\bu}$ connecting $x$ and $y$.
Hence $L_2(\mathbf u)$ has an additive subterm of the form
\[
x_1x_2+x_2x_3+\cdots+x_{2m+1}x_{2m+2}
\]
for some $m\geq 1$, where $x_1=x$ and $x_{2m+2}=y$.
Let $(A, B)$ be a bipartition of $\mathbb G_{\mathbf u}$. Then either $x\in A$, $y\in B$ or $x\in B$, $y\in A$.
Without loss of generality, assume that $x\in A$ and $y\in B$.
Define a substitution
$\psi:P_f(X_c^+)\to P_f(X_c^+)$ by
\[
\psi(t)=
\begin{cases}
y_1,&t=x_1,\\
y_4,&t=x_{2m+2},\\
y_3,&t\in A\setminus\{x_1\},\\
y_2,&t\in B\setminus\{x_{2m+2}\},\\
y_2y_3,&\text{otherwise}.
\end{cases}
\]
Then $\psi(\mathbf u)=y_1y_2+y_2y_3+y_3y_4$ and $\psi(\mathbf q)=y_1y_4$.
Hence $\psi(\mathbf q)\preceq \psi(\mathbf u)$ is precisely the identity \eqref{eq31}.
Since $\mathcal V$ satisfies $\mathbf q\preceq\mathbf u$,
it also satisfies $\psi(\mathbf q) \preceq \psi(\mathbf u)$, hence \eqref{eq31}.
Corollary~\ref{prop:define-S5-ST7} then implies that $\mathcal V$ is a subvariety of $\mathsf{V}(S_{(5, 2269)})$.
Since the reverse inclusion is already known, we obtain $\mathcal V = \mathsf{V}(S_{(5,2269)})$.
Therefore, $\pi^{-1}(\mathsf{V}(S_{(5,2269)}))
 =\{\mathsf{V}(S_{(5,2269)}),\mathsf{V}(SR_7),
 \mathsf{V}(ST_7)\}$.
\end{proof}

We now state the main theorem of this section.

\begin{theorem}\label{thm:ST7-lattice}
The subvariety lattice of $\mathsf{V}(ST_7)$ consists of exactly nine varieties:
$\mathsf{V}(ST_7)$, $\mathsf{V}(TR_6)$, $\mathsf{V}(SR_7)$, $\mathsf{V}(SR_6)$, $\mathsf{V}(S_{(5,2269)})$,
$\mathsf{V}(S_{(4,369)})$, $\mathsf{V}(S_c(ab))$, $\mathsf{V}(T_2)$, and $\mathbf T$.
Among these, $\mathsf{V}(ST_7)$, $\mathsf{V}(TR_6)$, $\mathsf{V}(SR_7)$, and $\mathsf{V}(SR_6)$ are nonfinitely based;
the remaining five are finitely based.
Its Hasse diagram is shown in Figure~\ref{fig:ST7-lattice}.

\begin{figure}[ht]
\centering
\begin{tikzpicture}[
scale=0.6,
  transform shape,
  x=1cm,
  y=1cm,
  cover/.style={draw=black!90,line width=0.95pt,line cap=round,line join=round},
  vertex/.style={circle,fill=white,draw=black,minimum size=7pt,inner sep=0pt,outer sep=0pt},
  lab/.style={font=\normalsize,inner sep=1pt,outer sep=0pt}
]
% Revised according to the user's sketch:
% - the three side segments are kept vertical/parallel;
% - N7 is placed exactly at the midpoint of the segment from ST7 to S_(5,2269).
\coordinate (ST7)    at ( 0.0, 8.8);
\coordinate (S52269) at ( 0.0, 4.3);
\coordinate (N7)     at ($(ST7)!0.5!(S52269)$);
\coordinate (TR6)    at (-3.0, 6.55);
\coordinate (S4369)  at (-3.0, 2.6);
\coordinate (SR6)    at ( 3.0, 4.75);
\coordinate (Scab)   at ( 3.0, 2.6);
\coordinate (T2)     at ( 0.0, 0.9);
\coordinate (T)      at ( 0.0,-2.1);

% Cover relations.
\draw[cover] (ST7)    -- (TR6);
\draw[cover] (ST7)    -- (N7);
\draw[cover] (TR6)    -- (S4369);
\draw[cover] (N7)     -- (S52269);
\draw[cover] (N7)     -- (SR6);
\draw[cover] (S4369)  -- (S52269);
\draw[cover] (S52269) -- (Scab);
\draw[cover] (SR6)    -- (Scab);
\draw[cover] (S4369)  -- (T2);
\draw[cover] (Scab)   -- (T2);
\draw[cover] (T2)     -- (T);

% Vertices.
\node[vertex] at (ST7)    {};
\node[vertex] at (TR6)    {};
\node[vertex] at (N7)     {};
\node[vertex] at (S52269) {};
\node[vertex] at (SR6)    {};
\node[vertex] at (S4369)  {};
\node[vertex] at (Scab)   {};
\node[vertex] at (T2)     {};
\node[vertex] at (T)      {};

% Labels placed in open regions.
\node[lab, anchor=south] at ($(ST7)+(0,0.42)$) {$\mathsf{V}(ST_7)$};
\node[lab, anchor=east]  at ($(TR6)+(-0.38,0.02)$) {$\mathsf{V}(TR_6)$};
\node[lab, anchor=west]  at ($(N7)+(0.36,0.10)$) {$\mathsf{V}(SR_7)$};
\node[lab, anchor=west]  at ($(SR6)+(0.36,0.00)$) {$\mathsf{V}(SR_6)$};
\node[lab, anchor=east]  at ($(S4369)+(-0.40,0.02)$) {$\mathsf{V}(S_{(4,369)})$};
\node[lab, anchor=west]  at ($(Scab)+(0.38,0.02)$) {$\mathsf{V}(S_c(ab))$};
\node[lab, anchor=west]  at ($(S52269)+(0.36,0.06)$) {$\mathsf{V}(S_{(5,2269)})$};
\node[lab, anchor=east]  at ($(T2)+(-0.34,0.00)$) {$\mathsf{V}(T_2)$};
\node[lab, anchor=north] at ($(T)+(0,-0.36)$) {$\mathbf{T}$};
\end{tikzpicture}
\caption{The subvariety lattice of $\mathsf{V}(ST_7)$.}
\label{fig:ST7-lattice}
\end{figure}
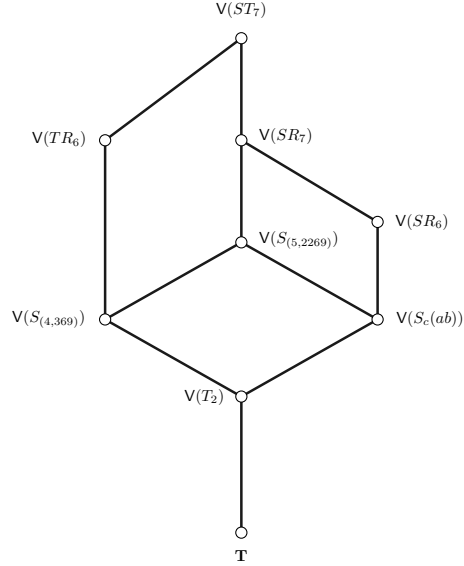
\end{theorem}
\begin{proof}
This follows from Proposition~\ref{prop:ST7-fibres}, Lemma~\ref{lem26090180}, Theorem~\ref{limit-variety}, Proposition~\ref{prop:basis-S5},  and Corollaries~\ref{cor:N7-nfb} and~\ref{cor:C7-nfb}.
\end{proof}

%\begin{remark}\label{rem:ST7-lattice-shape}
%The lattice $\operatorname{Sub}(\mathsf{V}(ST_7))$ is not distributive.
%Indeed,
%\[
%\{\mathsf{V}(S_{(4,369)}),\mathsf{V}(S_{(5,2269)}),\mathsf{V}(SR_7),
%\mathsf{V}(TR_6),\mathsf{V}(ST_7)\}
%\]
%is a sublattice isomorphic to the pentagon lattice $N_5$. Moreover, the
%interval
%\[
%[\mathsf{V}(T_2),\mathsf{V}(S_{(5,2269)})]
%=\{\mathsf{V}(T_2),\mathsf{V}(S_{(4,369)}),
%\mathsf{V}(S_c(ab)),\mathsf{V}(S_{(5,2269)})\}
%\]
%is a four-element Boolean lattice, that is, a Boolean square.
%\end{remark}

\begin{corollary}
The varieties $\mathsf{V}(SR_6)$ and $\mathsf{V}(TR_6)$ are the only limit subvarieties of $\mathsf{V}(ST_7)$.
\end{corollary}
\begin{proof}
This is a direct consequence of Theorem~\ref{thm:ST7-lattice}.
\end{proof}

\section{Conclusion}
We have shown that, for finite ai-semirings, being nonfinitely based does not imply being SNFB:
the ai-semirings $SR_6$ and $TR_6$ are the first two counterexamples.
It remains an open problem to determine, in general, when a finite nonfinitely based ai-semiring is SNFB.
We have also completely classified the subvarieties of the larger variety $\mathsf{V}(SR_6, TR_6)$, showing that it has exactly nine subvarieties, four of which are nonfinitely based and the remaining five finitely based.

From the proof of Theorem~\ref{thm:MR6}, $TR_6$ lies in $\mathsf{V}(S_{53})$.
Now let $I$ denote the set $\{k \in \mathbf{N} \mid k \geq 2\}$.
Then $I$ is an additive order filter and a multiplicative ideal of $\mathbf{N}$,
and so the corresponding Rees congruence $\rho_{_I}$ is a semiring congruence on $\mathbf{N}$.
It is easily verified that the quotient algebra $\mathbf{N}/\rho_{_I}$ is isomorphic to $S_{53}$.
Consequently, $S_{53}$ belongs to $\mathsf{V}(\mathbf{N})$.
Therefore, $\mathsf{V}(TR_6)$ is a limit subvariety of $\mathsf{V}(\mathbf{N})$.

The explicit examples of limit varieties of ai-semirings previously known from the literature are of several types.
Ren et al.~\cite{RenJacksonZhaoLei2023} constructed two types:
one consists of varieties generated by flat extensions of certain groups,
and the other is the variety $\mathsf{V}(S_c(abc))$, where $S_c(abc)$ is a certain eight-element ai-semiring.
Gao and Ren~\cite{GaoRen2025} gave another such example, which is not finitely generated.
Lyu et al.~\cite{LyuRenYue2026} recently showed that $\mathsf{V}(SR_6)$ is a limit variety.
None of these previously known examples lies in $\mathsf{V}(\mathbf{N})$.
Indeed, they all fail the inequality $x+y\preceq xy$, which is satisfied by both $\mathbf{N}$ and $TR_6$.

Thus $\mathsf{V}(TR_6)$ is a new limit variety of ai-semirings, and
provides the first explicit limit subvariety of $\mathsf{V}(\mathbf{N})$, addressing an open direction raised by Ren et al.~\cite[Section~6]{RenJacksonZhaoLei2023}:
``Understanding further limit varieties remains of interest. In this regard it would be
particularly interesting to explore subvarieties of the variety generated by the max-plus
algebra, toward possible explicit identification of a limit subvariety.''
At present, it is not known whether $\mathsf{V}(\mathbf{N})$ has any other limit subvarieties.

Since $SR_6$ and $TR_6$ share the same multiplicative reduct, it is natural to ask how many additions can be defined on this semigroup so that the resulting algebras become ai-semirings.
In fact, there are six such additions; the present paper has investigated two of them, namely $SR_6$ and $TR_6$.
The finite basis problem and subvariety lattices of the varieties generated by the remaining four algebras await further investigation,
and we hope to address them in future work.

\subsection*{Acknowledgment}
The authors would like to thank Mengya Yue for her helpful discussions related to this work. 
They also thank Professor Marcel Jackson for his valuable comments on an earlier version of the paper.
Miaomiao Ren, corresponding author, is supported by National Natural Science Foundation of China (12371024, 12571020).

\end{document}